\documentclass[12pt, a4paper]{amsart}
\usepackage[hmargin=30mm, vmargin=25mm, includefoot, twoside]{geometry}
\usepackage[bookmarksopen=true]{hyperref}

\usepackage{amsfonts,amssymb,verbatim}
\usepackage{latexsym}
\usepackage{mathrsfs}
\usepackage{stmaryrd}
\usepackage{xspace}
\usepackage{enumerate, paralist}

\usepackage[usenames,dvipsnames]{color}
 \usepackage{txfonts, pxfonts}

\usepackage{amsthm}
\usepackage{amsmath}

\newtheorem{thm}{Theorem}[section]
 \newtheorem{cor}[thm]{Corollary}
 \newtheorem{lem}[thm]{Lemma}
 \newtheorem{prop}[thm]{Proposition}
 \newtheorem{ex}[thm]{Example}
 \newtheorem{conj}[thm]{Conjecture}

 \theoremstyle{definition}
  \newtheorem{defn}[thm]{Definition}
 \theoremstyle{remark}
 \newtheorem{rem}[thm]{Remark}
 \def\asdim{\rm asdim\, }
 \def\diam{\mathrm{diam}\,}
 \def\ad{\mathrm{ad}}

\def\rank{\eta}

\begin{document}

\title{A characterization for asymptotic dimension growth}

\author{Goulnara Arzhantseva}
\address{Universit\"at Wien, Fakult\"at f\"ur Mathematik\\
Oskar-Morgenstern-Platz 1, 1090 Wien, Austria.}
\email{goulnara.arzhantseva@univie.ac.at}

\author{Graham A. Niblo}
\author{Nick Wright}
\author{Jiawen Zhang}
\address{School of Mathematics, University of Southampton, Highfield, SO17 1BJ, United Kingdom.
}
\email{g.a.niblo@soton.ac.uk,wright@soton.ac.uk,jiawen.zhang@soton.ac.uk}

\date{}

\thanks{Partially supported by the European Research Council (ERC) grant of Goulnara ARZHANTSEVA, no. 259527 and the Sino-British Fellowship Trust by Royal Society.}
\baselineskip=16pt

\begin{abstract}
  We give a characterization for asymptotic dimension growth. We apply it to CAT(0) cube complexes of finite dimension, giving an alternative proof of N.~Wright's result
  on their finite asymptotic dimension. We also apply our new characterization to geodesic coarse median spaces of finite rank and establish that they have subexponential asymptotic dimension growth.
  This strengthens a recent result of J.~\u{S}pakula and N.~Wright.
\end{abstract}
\maketitle

\section{Introduction}
The concept of asymptotic dimension was first introduced by Gromov~\cite{gromov1992asymptotic} in 1992 as a coarse analogue of the classical topological covering dimension. It started to attract much attention in 1998 when Yu  proved that the Novikov higher signature conjecture holds for groups with finite asymptotic dimension (FAD)~\cite{Yu98}. A lot of groups and spaces are known to have finite asymptotic dimension. Among those are, for instance, finitely generated abelian groups, free groups of finite rank, Gromov hyperbolic groups \cite{gromov1987hyperbolic, roe2005hyperbolic}, mapping class groups \cite{bestvina2010asymptotic}, CAT(0) cube complexes of finite dimension \cite{wright2012finite}, see \cite{BD05} for an excellent survey of these and other results. Recently Behrstock, Hagen and Sisto introduced the powerful new notion of hierarchically hyperbolic spaces and showed that these have finite asymptotic dimension \cite{behrstock2015asymptotic}, recovering a number of the above results, including notably mapping class groups and a number of CAT(0) cube complexes.

On the other hand, there are many groups and spaces with infinite asymptotic dimension. Examples are the wreath product $\mathbb{Z} \wr \mathbb{Z}$, the Grigorchuk group \cite{smith2007asymptotic}, the Thompson groups, etc. Generalizing FAD, Dranishnikov 
 defined the asymptotic dimension growth for a space \cite{dranishnikov2006groups}; if the asymptotic dimension growth function is eventually constant then the space has FAD. Dranishnikov showed that the wreath product of a finitely generated nilpotent group with a finitely generated FAD group has polynomial asymptotic dimension growth. He also showed that polynomial asymptotic dimension growth implies Yu's Property A, and, hence, the coarse Baum-Connes Conjecture, provided the space has bounded geometry \cite{yu2000coarse}. 
 Later, Ozawa \cite{ozawa2012metric}  extended this result to spaces with subexponential growth; see also~\cite{oppenheim2014depth}.
Bell analyzed how the asymptotic dimension function is affected by various group-theoretical constructions~\cite{bell2005growth}.

In this paper, we give an alternative characterization for the asymptotic dimension growth function which is inspired by Brown and Ozawa's proof of Property~A for Gromov's hyperbolic groups, ~\cite[Theorem 5.3.15]{brown2008c},  which is in turn
inspired by \cite{kaimanovich2004boundary}. We use this to study two notable examples: CAT(0) cube complexes of finite dimension and coarse median spaces of finite rank.

The techniques used to study these examples are developments of those used by \u{S}pakula and Wright \cite{spakula2016coarse} to establish Property A for uniformly locally finite coarse median spaces of finite rank. As a byproduct, we obtain a new proof of finite asymptotic dimension for CAT(0) cube complexes which allows one to explicitly construct the required controlled covers. This compares with Wright's original proof, \cite{wright2012finite},  which is discussed below.

CAT(0) cube complexes are a nice class of non-positively curved spaces, first studied by Gromov who gave a purely combinatorial
condition for recognizing the non-positive curvature of cube complexes~\cite{gromov1987hyperbolic}. Many well-known groups act properly on CAT(0) cube complexes. For instance, right-angled Artin groups, many small cancellation groups, and Thompson's groups admit such actions.
This makes it possible to deduce properties of these groups from the corresponding properties of the CAT(0) cube complexes.

In 2010, Wright \cite{wright2012finite} proved that the asymptotic dimension of a CAT(0) cube complex $X$ is bounded by its dimension. He proved this by constructing a family of  $\varepsilon-$Lipschitz cobounded maps to CAT(0) cube complexes of (at most) the same dimension indexed by $\varepsilon>0$. We  use our characterization for finite asymptotic dimension to give a direct proof of this result. Namely, we construct uniformly bounded covers with suitable properties. Being more explicit, this proof loses, however, the sharp bound on the asymptotic dimension. Thus, we give an alternative proof of  the following non-quantitative variant of Wright's theorem:
\begin{thm}\label{thm a}
Let $X$ be a CAT(0) cube complex of finite dimension, then $X$ has finite asymptotic dimension.
\end{thm}

The key point in our approach is to analyse the normal cube path distance on the cube complex, introduced by Niblo and Reeves~\cite{niblo1998geometry}. We consider the ball with respect to the normal cube path distance rather than to the ordinary edge-path distance. We decompose such a ball into intervals and use induction on the dimension in order to construct some ``separated" net satisfying a suitable consistency property. In the process, we give a detailed analysis of normal balls and normal spheres (i.e. balls and spheres with respect to the normal cube path distance). See Section 4 for all details.

Our second application is to coarse median spaces. They were introduced by Bowditch as a coarse variant of classical median spaces~\cite{bowditch2013coarse}. The notion of a coarse median group leads to a unified viewpoint on several interesting classes of groups, including Gromov's hyperbolic groups, mapping class groups, and CAT(0) cubical groups. Bowditch  showed that hyperbolic spaces are exactly coarse median spaces of rank 1, and mapping class groups are examples of coarse median spaces of finite rank \cite{bowditch2013coarse}. He  also established interesting properties for coarse median spaces such as Rapid Decay, the property of having quadratic Dehn function, etc.

Intuitively, a coarse median space is a metric space equipped with a ternary operator (called the coarse median), in which every finite subset can be approximated by a finite median algebra. In these approximations the coarse median is approximated by an actual median with the distortion being controlled by the metric. This extends Gromov's observation that in a $\delta$-hyperbolic space finite subsets can be well approximated by finite trees.

Recently,  \u{S}pakula and Wright  proved that a coarse median space with finite rank and at most exponential volume growth has Property A~\cite{spakula2016coarse}. Following their proof and using our characterization for asymptotic dimension growth, we obtain the following result:
\begin{thm}\label{thm b}
Let $X$ be a geodesic coarse median space with finite rank and  at most exponential volume growth, then $X$ has subexponential asymptotic dimension growth.
\end{thm}

 Hierarchically hyperbolic spaces are examples of coarse median spaces, see \cite{behrstock2015hierarchically}, hence our theorem is broader in scope, though with a weaker conclusion, than the finite asymptotic dimension result proven in \cite{behrstock2015asymptotic}. We expect the following
general result.
\begin{conj}
Every geodesic coarse median space with finite rank has finite asymptotic dimension.
\end{conj}
By a result of Ozawa \cite{ozawa2012metric}, subexponential asymptotic dimension growth implies Property A, thus, our theorem strengthens  the result of \cite{spakula2016coarse}.

The paper is organized as follows. In Section 2, we give some preliminaries on asymptotic dimension growth, CAT(0) cube complexes, and coarse median spaces. In Section 3, we provide a characterization of the asymptotic dimension growth function, and, as a special case, give a characterization of finite asymptotic dimension. Sections 4 and 5 deal with CAT(0) cube complexes: in Section 4, we study normal balls and spheres which are essential in our approach to prove Theorem \ref{thm a} in Section 5. Section 6 deals with the coarse median case, and we prove Theorem \ref{thm b} there.

\section{Preliminaries}
\subsection{Asymptotic Dimension}
The notion of asymptotic dimension was first introduced by Gromov in 1993 \cite{gromov1992asymptotic} as a coarse analogue of the classical Lebesgue topological covering dimension. See also \cite{BD05}.

Let $(X,d)$ be a metric space and $r>0$. We call a family $\mathcal{U}=\{U_i\}$ of subsets in $X$ $r-$\emph{disjoint}, if for any $U\neq U'$ in $\mathcal{U}$, $d(U,U')\geqslant r$, where $d(U,U')=\inf\,\{d(x,x'):x\in U,x'\in U'\}$. We write
$$\bigsqcup_{r-disjoint}U_i$$
for the union of $\{U_i\}$. A family $\mathcal{V}$ is said to be \emph{uniformly bounded}, if $\mathrm{mesh}(\mathcal V)=\sup\,\{\mathrm{diam}(V):V\in\mathcal{V}\}$ is finite.
Let $\mathcal{U}=\{U_i\}$ be a cover of $X$ and $r> 0$. We define the \emph{$r-$multiplicity} of $\mathcal{U}$, denoted by $m_r(\mathcal{U})$,
 to be the minimal integer $n$ such that for any $x\in X$, the ball $B(x,r)$ intersects at most $n$ elements of $\mathcal{U}$.
As usual, $m(\mathcal U)$ denotes the \emph{multiplicity} of a cover $\mathcal U$, that is
 the maximal number of elements of $\mathcal U$ with a non-empty intersection.
 A number $\lambda>0$ is called a \emph{Lebesgue number} of $\mathcal{U}$, if for every subset $A \subseteq X$ with diameter $\leqslant \lambda$, there exists an element $U \in \mathcal{U}$ such that $A \subseteq U$. The Lebesgue number $L(\mathcal{U})$ of the cover $\mathcal{U}$ is defined to be the infimum of all Lebesgue numbers of $\mathcal{U}$.

\begin{defn}[\cite{gromov1992asymptotic}]
We say that the \emph{asymptotic dimension} of a metric space $X$ does not exceed $n$ and we write $\asdim X\leqslant n$, if for every $r>0$, the space $X$ can be covered by $n+1$ subspaces $X_0,X_1,\ldots,X_n$, and each $X_i$ can be further decomposed into some $r-$disjoint uniformly bounded subspaces:
$$X=\bigcup^n_{i=0}X_i,\mbox{\quad} X_i=\bigsqcup_{\mbox{\scriptsize$\begin{array}{c} r-disjoint \\ j\in \mathbb{N} \end{array}$}}X_{ij}, \hbox{ and }\sup_{i,j}\diam X_{ij}<\infty.$$
We say $\asdim X=n$, if $\asdim X\leqslant n$ and $\asdim X$ is not less than $n$.
\end{defn}

Here are basic examples of spaces and groups with finite asymptotic dimension.
\begin{ex}[\cite{NY12}, \cite{roe2005hyperbolic}]
\strut
\begin{enumerate}[1)]
  \item $\asdim \mathbb{Z}^n$ = $n$ for all $n\in \mathbb{N}$, where $\mathbb{Z}$ is the group of integers;
  \item Gromov's $\delta$-hyperbolic spaces, e.g., word hyperbolic groups, have finite asymptotic dimension.
\end{enumerate}
\end{ex}

From the definition, it is easy to see that the asymptotic dimension of a subspace is at most that of the ambient space. There are other equivalent definitions of asymptotic dimension. We list one here for a later use, and guide the reader to \cite{BD05} for  others.
\begin{prop}[\cite{BD05}]
Let $X$ be a metric space, then $\asdim X \leqslant n$ if and only if for any $r>0$, there exists a uniformly bounded cover $\mathcal{U}$ of $X$, such that $m_r(\mathcal{U}) \leqslant n+1$.
\end{prop}

\subsection{Asymptotic Dimension Growth}
Let us consider the direct sum of infinitely many copies of the integers: $G=\bigoplus\limits_\infty \mathbb{Z}$. Since for any $n \in \mathbb{N}$, the group $\mathbb{Z}^n$ is contained in $G$, by the above mentioned results, $G$ has infinite asymptotic dimension. In order to deal with such groups/spaces, Dranishnikov studied the following concept as a generalization of the property of having a finite asymptotic dimension.

\begin{defn}[\cite{dranishnikov2006groups}]\label{def for asdim growth}
Let $(X,d)$ be a metric space. Define a function
$$\ad_X(\lambda)=\min\{ m(\mathcal{U}):\mathcal{U} \mbox{~is~a~cover~of~}X,L(\mathcal{U}) > \lambda\}-1,$$
which is called the \emph{asymptotic dimension function} of $X$.
\end{defn}
Note that $ad_X$ is monotonic and
$$\lim_{\lambda \rightarrow \infty} ad_X(\lambda)=\asdim(X).$$

Like in the case of the volume function, the growth type of the asymptotic dimension function is more essential than the function itself. Recall that for $f,g\colon\mathbb{R}_+ \rightarrow \mathbb{R}_+$, we write $f \preceq g$, if there exists $k\in \mathbb{N}$, such that for any $x>k$, $f(x) \leqslant kg(kx+k)+k$. We write $f\thickapprox g$ if both $f \preceq g$ and $g \preceq f$. It is clear that
$``\thickapprox"$ is an equivalence relation. We define the growth type of $f$ to be the $\thickapprox$-equivalence class of $f$. Define the asymptotic dimension growth of $X$ to be the growth type of $\ad_X$.

By a result of Bell and Dranishnikov, the growth type of the asymptotic dimension function is a quasi-isometric invariant.
\begin{prop}[\cite{bell2005growth,dranishnikov2006groups}]
Let $X$ and $Y$ be two discrete metric spaces with bounded geometry. If $X$ and $Y$ are quasi-isometric, then $ad_X \approx ad_Y$. In particular, the asymptotic dimension growth is well-defined for finitely generated groups.
\end{prop}

We give an alternative (equivalent) definition of the asymptotic dimension growth that is used in our characterization.
\begin{lem}\label{def for asdim growth2}
Let $X$ be a metric space, and define
$$\widetilde{\ad}_X(\lambda)=\min\{\, m_\lambda(\mathcal{U}):\mathcal{U} \mbox{~is~a~cover~of~}X\}-1.$$
Then $\widetilde{\ad}_X \approx \ad_X$.
\end{lem}

\begin{proof}
Given $\lambda>0$, suppose $\mathcal{U}$ is a cover of $X$ with $L(\mathcal{U}) > \lambda$. For any $U\in \mathcal{U}$, define the \emph{inner $\lambda-$neighborhood} of $U$ to be
$$N_{-\lambda}(U)=X\setminus N_\lambda(X \setminus U),$$
where $N_\lambda$ denotes the usual $\lambda$-neighborhood of the set,
and we define $$N_{-\lambda}(\mathcal{U})=\{N_{-\lambda}(U): U \in \mathcal{U}\}.$$ Since $L(\mathcal{U}) > \lambda$, $N_{-\lambda}(\mathcal{U})$ is still a cover of $X$. By definition, it is obvious that $m_\lambda(N_{-\lambda}(\mathcal{U})) \leqslant m(\mathcal{U})$, which yields $\widetilde{\ad}_X \preceq \ad_X$.

Conversely suppose $\mathcal{U}$ is a cover of $X$. Consider $N_\lambda(\mathcal{U})$, which has Lebesgue number not less than $\lambda$. It is easy to show $m(N_\lambda(\mathcal{U})) \leqslant m_\lambda (\mathcal{U})$, which implies $\ad_X \preceq \widetilde{\ad}_X$.
\end{proof}

By the preceding lemma, we can use either $\ad_X$ or $\widetilde{\ad}_X$ as the definition for the asymptotic dimension function. Recall  that if there exists a polynomial (subexponential) function $f$ such that $\ad_X \preceq f$, then $X$ is said to have polynomial (subexponential) asymptotic dimension growth.

Dranishnikov has shown that polynomial asymptotic dimension growth implies Yu's Property A, and he gave a class of groups having such property.
\begin{prop}[\cite{dranishnikov2006groups}]
Let $N$ be a finitely generated nilpotent group and $G$ be a finitely generated group with finite asymptotic dimension. Then the wreath product $N \wr G$ has polynomial asymptotic dimension growth. In particular, $\mathbb{Z} \wr \mathbb{Z}$ has polynomial asymptotic dimension growth.
\end{prop}

\subsection{CAT(0) Cube Complexes}
We recall basic notions and results on the structure of CAT(0) cube complexes. We omit some details and most of the proofs but direct the readers to \cite{BH99, chepoi2000graphs, gromov1987hyperbolic, niblo1998geometry, sageev1995ends} for more information.

A \emph{cube complex} is a polyhedral complex in which each cell is isometric to a Euclidean cube and the gluing maps are isometries. The \emph{dimension} of the complex is the maximum of the dimensions of the cubes. For a cube complex $X$, we can associate it with the intrinsic pseudo-metric $d_{int}$, which is the minimal pseudo-metric on $X$ such that each cube embeds isometrically. When $X$ has finite dimension, $d_{int}$ is a complete geodesic metric on $X$. See \cite{BH99} for a general discussion on polyhedral complex and the associated intrinsic metric.

There is also another metric associated with $X$. Let  $X^{(1)}$ be the 1-skeleton of $X$, that is a graph with the vertex set $V=X^{(0)}$. We equip $V$ with the edge-path metric $d$, which is the minimal number of edges in a path connecting two given vertices. Clearly, when $X^{(1)}$ is connected, $d$ is a geodesic metric on $V$.
For $x,y\in V$,  the \emph{interval} is defined by $[x,y]=\{z\in V:d(x,y)=d(x,z)+d(x,y)\}$, that is it consists of all points on any geodesic between $x$ and $y$.

A geodesic metric space $(X,d)$ is \emph{CAT(0)} if all geodesic triangles in $X$ are slimmer than the comparative triangle in the Euclidean space. For a cube complex $(X,d_{int})$, Gromov has given a combinatorial characterization of the CAT(0) condition~\cite{gromov1987hyperbolic}: $X$ is CAT(0) if and only if it is simply connected and the link of each vertex is a flag complex (see also \cite{BH99}).

Another characterization of the CAT(0) condition was obtained by Chepoi \cite{chepoi2000graphs} (see also~\cite{roller1998poc}): a cube complex $X$ is CAT(0) if and only if for any $x,y,z\in V$, the intersection $[x,y] \cap [y,z] \cap [z,x]$ consists of a single point $\mu(x,y,z)$, which is called the \emph{median} of $x,y,z$. In this case, we call the graph $X^{(1)}$ a \emph{median graph}; and $V$ equipped with the ternary operator $m$ is indeed a median algebra~\cite{isbell1980median}. In particular, the following equations hold: $\forall x,y,z,u,v \in V$,
\begin{itemize}
  \item[M1.] $\mu(x,x,y)=x$;
  \item[M2.] $\mu(\sigma(x),\sigma(y),\sigma(z))=\mu(x,y,z)$, where $\sigma$ is any permutation of $\{x,y,z\}$;
  \item[M3.] $\mu(\mu(x,y,z),u,v)=\mu(\mu(x,u,v),\mu(y,u,v),z)$.
\end{itemize}
Obviously, $\mu(x,y,z)\in [x,y]$, and $[x,y]=\{z\in V: \mu(x,y,z)=z\}$.
\begin{lem}\label{basic median}
Let $x,y,z,w \in V$ such that $z,w\in [x,y]$. Then $z\in[x,w]$ implies $w \in [z,y]$.
\end{lem}

\begin{proof}
Since $z\in [x,w]$ and $w\in [x,y]$, we have $\mu(z,x,w)=z$ and $\mu(x,w,y)=w$. So,
$\mu(z,w,y)=\mu(\mu(z,x,w),w,y)=\mu(\mu(z,w,y),\mu(x,w,y),w)=\mu(\mu(z,w,y),w,w)=w,$
which implies $w\in [z,y]$.
\end{proof}

\begin{lem}\label{weakly modular}
For $x,y,z \in V$ and $d(z,y)=1$, $[x,z] \subseteq [x,y]$ or $[x,y] \subseteq [x,z]$.
\end{lem}

\begin{proof}
By Chepoi's result~\cite{chepoi2000graphs}, $X^{(1)}$ is a median graph, hence it is weakly modular (see~\cite{chepoi2000graphs}). So $d(x,y)\neq d(x,z)$, which implies $d(x,y)=d(x,z)+1$ or $d(x,z)=d(x,y)+1$, i.e. $[x,z] \subseteq [x,y]$ or $[x,y] \subseteq [x,z]$.
\end{proof}

A CAT(0) cubical complex $X$ can be equipped with a set of \emph{hyperplanes}~\cite{chatterji2005wall, niblo1998geometry, nica2004cubulating, sageev1995ends}. Each hyperplane does not intersect itself, and divides the space into two halfspaces. Given two hyperplanes $h,k$, if the four possible intersections of halfspaces are all nonempty, then we say $h$ crosses $k$, denoted by $h\pitchfork k$. This occurs if and only if $h$ and $k$ cross a common cube $C$ (also denoted by $h\pitchfork C$). Furthermore~\cite{sageev1995ends}, given a maximal collection of pairwise intersecting hyperplanes, there exists a unique cube which all of them cross. Thus, the dimension of $X$ is the maximal number of pairwise intersecting hyperplanes. We can also define intervals in the language of hyperplanes: $[x,y]$ consists of points which lie in all halfspaces containing $x$ and $y$.

We call a subset $Y \subseteq V$ \emph{convex}, if for any $x,y\in Y$, $[x,y]\subseteq Y$. Obviously, halfspaces are convex since any geodesic crosses a hyperplane at most once~\cite{niblo1998geometry, sageev1995ends}. This also implies
$$d(x,y)= \sharp \{\mbox{~hyperplane~}h: h \mbox{~separates~} x \mbox{~from~} y\}.$$

\subsection{Coarse Median Spaces}
According to Gromov, hyperbolic spaces can be considered locally as a coarse version of trees, in the sense that every finite subset can be approximated by a finite tree in a controlled way~\cite{gromov1987hyperbolic}. If one wants to approximate a space locally by finite median algebras (graphs), this would turn to the definition of coarse median spaces introduced by Bowditch. See~\cite{bowditch2013coarse, bowditch2014embedding, zeidler2013coarse} for details.

\begin{defn}[\cite{bowditch2013coarse}]
Let $(X,\rho)$ be a metric space, and $\mu\colon X^3 \rightarrow X$ be a ternary operation. We say that $(X,\rho,\mu)$ is a \emph{coarse median space} and $\mu$ is a \emph{coarse median} on $X$, if the following conditions hold:
\begin{itemize}
  \item[C1.] There exist constants $K, H(0)>0$ such that $\forall a,b,c,a',b',c' \in X,$
        $$ \rho(\mu(a,b,c), \mu(a',b',c')) \leqslant K(\rho(a,a')+\rho(b,b')+\rho(c,c'))+H(0).$$
  \item[C2.] There exists a function $H\colon  \mathbb{N} \rightarrow [0,+\infty)$ with the following property. For a finite subset $A\subseteq X$ with $1 \leqslant |A| \leqslant p$, there exists a finite median algebra $(\Pi, \rho_{\Pi}, \mu_{\Pi})$, and maps $\pi\colon A \rightarrow \Pi$, $\lambda\colon  \Pi \rightarrow X$ such that $\forall x,y,z \in \Pi, a\in A$,
      $$\rho(\lambda \mu_{\Pi}(x,y,z), \mu(\lambda x, \lambda y, \lambda z)) \leqslant h(p),$$
      and
      $$\rho(a, \lambda \pi a) \leqslant h(p).$$
\end{itemize}
We refer to $K,H$ as the \emph{parameters} of $(X,\rho,\mu)$.
Furthermore, if there exists $d \in \mathbb{N}$, such that we can always choose the median algebra $\Pi$ in condition (C2) above of rank at most $\rank$, then we say $X$ has (coarse) \emph{rank} at most $\rank$.

A finitely generated group is said to be \emph{coarse median} if some Cayley graph has a coarse median.
\end{defn}

Note that, by definition, a coarse median on a group is not required to be equivariant under the group action.

\begin{rem}\label{median assup}
According to Bowditch, without loss of generality, we may always assume that $\mu$ satisfies the median axioms M1 and M2: for all $a,b,c\in X$,
\begin{itemize}
  \item[M1.] $\mu(a,a,b)=a$;
  \item[M2.] $\mu(a,b,c)=\mu(b,c,a)=\mu(b,a,c)$.
\end{itemize}
\end{rem}

A large class of groups and spaces have been shown to be coarse median, including Gromov's hyperbolic groups, right-angled Artin groups, mapping class groups, CAT(0) cube complexes, etc.~\cite{bowditch2013coarse}. Bowditch has proved that coarse median groups have Property of Rapid Decay~\cite{bowditch2014embedding}, quadratic Dehn's function~\cite{bowditch2013coarse}, etc.
This yielded a unified way to prove these properties for the above-listed groups. Recently, \u{S}pakula and Wright have proved that coarse median spaces
of finite rank and of at most exponential volume growth have Yu's Property A~\cite{spakula2016coarse}.

\section{Characterization for asymptotic dimension growth}
In this section, we establish a characterization for asymptotic dimension growth and obtain several interesting consequences of this main result.
For instance, we get a characterization for a group to have finite asymptotic dimension.

\begin{thm}\label{main prop}
Let $(X,d)$ be a discrete metric space, and $f\colon \mathbb{R}_+ \rightarrow \mathbb{R}_+$ be a function. Then the following are equivalent.
\begin{enumerate}
  \item $\ad_X \preceq f$;
  \item There exists a function $g\colon \mathbb{R}_+ \rightarrow \mathbb{R}_+$ which has the same growth type as $f$, such that $\forall l \in \mathbb{N}$, $\forall k=1,2,\ldots, 3l$, $\forall x \in X$, we can assign a subset $S(x,k,l) \subseteq X$, satisfying:
        \begin{enumerate}[i)]
          \item $\forall l \in \mathbb{N}$, $\exists S_l >0$, such that $S(x,k,l) \subseteq B(x,S_l)$ for all $k=1,\ldots, 3l$;
          \item $\forall l \in \mathbb{N}$, $\forall k,k'$ with $1 \leqslant k \leqslant k' \leqslant 3l$, $\forall x \in X$,
                we have $S(x,k,l) \subseteq S(x,k',l)$;
          \item $\forall x,y \in X$ with $d(x,y) \leqslant l$, we have:
                \begin{itemize}
                  \item $S(x,k-d(x,y),l) \subset S(x,k,l) \cap S(y,k,l)$, for $k=d(x,y)+1,\ldots,3l$;
                  \item $S(x,k+d(x,y),l) \supset S(x,k,l) \cap S(y,k,l)$, for $k=1,\ldots,3l-d(x,y)$;
                \end{itemize}
          \item $\forall l \in \mathbb{N}$, $\forall k=1,2,\ldots,3l$, $\forall x \in X$, we have $\sharp S(x,k,l) \leqslant g(l)$.
        \end{enumerate}
\end{enumerate}
\end{thm}

\begin{proof}
~\\[-0.5cm]
\begin{itemize}
  \item $1 \Rightarrow 2$: By Lemma \ref{def for asdim growth2}, we can assume that there exists a function $g\colon \mathbb{R}_+ \rightarrow \mathbb{R}_+$ with $g\thickapprox f$ such that $\forall l \in \mathbb{N}$, there exists a uniformly bounded cover $\mathcal{U}$ of $X$ with $m_{3l}(\mathcal{U}) \leqslant g(l)$. Suppose $\mathcal{U} = \{U_i:i \in I\}$, and choose $x_i \in U_i$.
      For $k=1,2,\ldots, 3l$ and $x\in X$, we define
      $$S(x,k,l)=\{x_i: B(x,k) \cap U_i \neq \emptyset\}.$$
      Now let us check the four properties in Condition 2.
      \begin{enumerate}[i)]
        \item If $B(x,k) \cap U_i \neq \emptyset$, we can assume $y \in B(x,k) \cap U_i$. Now $d(y,x_i) \leqslant \mathrm{mesh}(\mathcal{U})$, so $d(x,x_i) \leqslant k+\mathrm{mesh}(\mathcal{U}) \leqslant 3l+\mathrm{mesh}(\mathcal{U})$. In other words,
            $$S(x,k,l) \subseteq B(x,3l+\mathrm{mesh}(\mathcal{U})).$$
        \item It is immediate, by our definition of sets $S(x,k,l)$.
        \item $\forall x,y \in X$ with $d(x,y) \leqslant l$, $\forall k=d(x,y)+1,\ldots,3l$, we have:
              $$S(x,k-d(x,y),l)=\{x_i: B(x,k-d(x,y)) \cap U_i \neq \emptyset\}.$$
              Now if $B(x,k-d(x,y)) \cap U_i \neq \emptyset$, we can assume $z \in B(x,k-d(x,y)) \cap U_i$, i.e. $z\in U_i$ and $d(z,x) \leqslant k-d(x,y)$. So $d(z,y) \leqslant k$, i.e. $z \in B(y,k) \cap U_i$. So $B(y,k) \cap U_i \neq \emptyset$, which implies
              $$S(x,k-d(x,y),l) \subset S(x,k,l) \cap S(y,k,l).$$

              On the other hand, $\forall k'=1,\ldots,3l-d(x,y)$, suppose $x_j \in S(x,k',l) \cup S(y,k',l)$. We can assume that $x_j \in S(y,k',l)$, i.e. $B(y,k')\cap U_j\neq \emptyset$, which implies $B(x,k'+d(x,y)) \cap U_j \neq \emptyset$. So we have:
              $$S(x,k'+d(x,y),l) \supset S(x,k',l) \cap S(y,k',l).$$
        \item $\forall l \in \mathbb{N}$, $\forall k=1,2,\ldots,3l$, $\forall x \in X$, we have
              $$\sharp S(x,k,l)=\sharp \{x_i: B(x,k) \cap U_i \neq \emptyset\} \leqslant m_{3l}(\mathcal{U}) \leqslant g(l).$$
      \end{enumerate}

  \item $2 \Rightarrow 1$: $\forall l \in \mathbb{N}$, let $H=\bigcup\limits_{x\in X}S(x,l,l)$. Also, $\forall h \in H$, we define $A_h=\{y: h \in S(y,l,l)\}$.
        We define $\mathcal{U}_l=\{A_h: h\in H\}$. Since $\forall x\in X$, if we take $h\in S(x,l,l)$, then $x\in A_h$. So $\mathcal{U}_l$ is a cover of $X$.
        Since $\exists S_l >0$ such that $S(x,l,l) \subseteq B(x,S_l)$, we know that $d(h,y) \leqslant S_l$ for all $y\in A_h$, which implies $\mathrm{mesh}(\mathcal{U}_l) \leqslant S_l$.
        Finally, let us analyse $m_l(\mathcal{U}_l)$. $\forall x \in X$, consider $h\in H$ with $B(x,l) \cap A_h \neq \emptyset$. Take $y\in B(x,l) \cap A_h$, i.e. $d(y,x) \leqslant l$ and $h \in S(y,l,l)$. Now by assumptions in Condition 2, we have
        $$S(y,l,l) \subseteq S(x,l+d(x,y),l) \subseteq S(x,2l,l).$$
        So $m_l(\mathcal{U}_l) \leqslant \sharp S(x,2l,l) \leqslant g(l)$. 
        Finally by Lemma \ref{def for asdim growth2}, we have 
        $$\ad_X \thickapprox \widetilde{\ad}_X \leqslant g \thickapprox f.$$
\end{itemize}

\end{proof}

Taking in the preceding theorem a constant function $f$, we obtain a characterization for finite asymptotic dimension.

\begin{cor}\label{char1}
Let $(X,d)$ be a discrete metric space, $n \in \mathbb{N}$. Then the following are equivalent:
\begin{enumerate}
  \item $\asdim X \leqslant n$;
  \item $\forall l \in \mathbb{N}$, $\forall k=1,2,\ldots, 3l$, $\forall x \in X$, we can assign a subset $S(x,k,l) \subseteq X$, satisfying:
        \begin{enumerate}[i)]
          \item $\forall l \in \mathbb{N}$, $\exists S_l >0$, such that $S(x,k,l) \subseteq B(x,S_l)$ for all $k=1,\ldots, 3l$;
          \item $\forall l \in \mathbb{N}$, $\forall k,k'$ with $1 \leqslant k \leqslant k' \leqslant 3l$, $\forall x \in X$,
                we have $S(x,k,l) \subseteq S(x,k',l)$;
          \item $\forall x,y \in X$ with $d(x,y) \leqslant l$, we have:
                \begin{itemize}
                  \item $S(x,k-d(x,y),l) \subset S(x,k,l) \cap S(y,k,l)$, for $k=d(x,y)+1,\ldots,3l$;
                  \item $S(x,k+d(x,y),l) \supset S(x,k,l) \cap S(y,k,l)$, for $k=1,\ldots,3l-d(x,y)$;
                \end{itemize}
          \item $\forall l \in \mathbb{N}$, $\forall k=1,2,\ldots,3l$, $\forall x \in X$, we have $\sharp S(x,k,l) \leqslant n+1$.
        \end{enumerate}
  \end{enumerate}
\end{cor}

Now we turn to the case when $X$ is a graph, and obtain a characterization for finite asymptotic dimension which is easier to check.
\begin{cor}\label{char for asdim}
Given a graph $X=(V,E)$ with vertices $V$ and edges $E$, and equipped with the edge-path length metric $d$, then the following are equivalent:
\begin{enumerate}
  \item $\asdim X \leqslant n$;
  \item $\forall l \in \mathbb{N}$, $\forall k=1,2,\ldots, 3l$, $\forall x \in X$, we can assign a subset $S(x,k,l) \subseteq X$, satisfying:
        \begin{enumerate}[i)]
          \item $\forall l \in \mathbb{N}$, $\exists S_l >0$, such that $S(x,k,l) \subseteq B(x,S_l)$ for all $k=1,\ldots, 3l$;
          \item $\forall l \in \mathbb{N}$, $\forall k,k'$ with $1 \leqslant k \leqslant k' \leqslant 3l$, $\forall x \in X$,
                we have $S(x,k,l) \subseteq S(x,k',l)$;
          \item $\forall x,y \in X$ with $d(x,y)=1$ (i.e. with $x$ and $y$ connected by an edge), we have:
                $S(y,k,l) \subseteq S(x,k+1,l)$ for all $k=1,2,\ldots, 3l-1$;
          \item $\forall l \in \mathbb{N}$, $\sharp S(x,2l,l) \leqslant n+1$.
        \end{enumerate}
\end{enumerate}
\end{cor}

\begin{rem}
The only distinction between the above two corollaries is that in Corollary~\ref{char for asdim}, assumption 2/iii) is required only for endpoints of an edge rather than for an arbitrary pair of points as in
Corollary~\ref{char1}. We point out that the preceding corollaries can be generalized to the case of arbitrary asymptotic dimension growth. We will not use such a generalization, so we omit it.
\end{rem}

\begin{proof}[Proof of Corollary 3.3]
\strut

$1 \Rightarrow 2$ is implied directly by Corollary \ref{char1}, so we focus on $2 \Rightarrow 1$.

Following the proof of $2 \Rightarrow 3$ in Proposition 3.1, $\forall l \in \mathbb{N}$, let $H=\bigcup\limits_{x\in X}S(x,l,l)$. And $\forall h \in H$, define $A_h=\{y: h \in S(y,l,l)\}$.
Define $\mathcal{U}_l=\{A_h: h\in H\}$. Since $\forall x\in X$, if we take $h\in S(x,l,l)$, then $x\in A_h$. So $\mathcal{U}_l$ is a cover of $X$.
Since $\exists S_l >0$ such that $S(x,l,l) \subseteq B(x,S_l)$, we know $d(h,y) \leqslant S_l$ for all $y\in A_h$, which implies $\mathrm{mesh}(\mathcal{U}_l) \leqslant S_l$.
Finally, let us analyse $m_l(\mathcal{U}_l)$. $\forall x \in X$, consider $h\in H$ with $B(x,l) \cap A_h \neq \emptyset$. Take $y\in B(x,l) \cap A_h$, i.e. $d(y,x) \leqslant l$ and $h \in S(y,l,l)$. By the definition of the edge-path length metric $d$, we know that there exists a sequence of vertices $y=y_0,y_1,\ldots,y_k=x$ such that $y_i \in V$ for all $i=0,1,\ldots,k$, $d(y_i,y_{i+1})=1$ for all $i=0,1,\ldots,k-1$, and $k\leqslant l$. Now by the hypothesis, we know:
$$S(y,l,l) \subseteq S(y_1,l+1,l) \subseteq S(y_2,l+2,l) \subseteq \ldots \subseteq S(y_k, l+k, l)=S(x,k+l,l)\subseteq S(x,2l,l).$$
So $\{h \in H: B(x,l) \cap A_h \neq \emptyset\} \subseteq S(x,2l,l)$, which implies $m_l(\mathcal{U}_l) \leqslant \sharp S(x,2l,l) \leqslant n+1$.
\end{proof}

\section{Normal cube path and normal distance}\label{sec:normal}
In the next two sections, we focus on CAT(0) cube complexes, and prove Theorem~\ref{thm a}. We prove it by constructing a uniformly bounded cover with suitable properties. Such a construction relies deeply on the analysis of normal balls and spheres, which we give in this section.

Normal cube paths, which were introduced by Niblo and Reeves in ~\cite{niblo1998geometry} play a key role in the construction of the cover. They determine a distance function on the vertices and the balls and spheres defined in terms of this distance are essential in our proof of Theorem~\ref{thm a}.

Throughout this section we fix a CAT(0) cube complex $X$ with a fixed vertex $x_0$. The 1-skeleton $X^{(1)}$ of $X$ is a graph with vertex set $V=X^{(0)}$, and edge set $E$, which give us the edge metric $d$ on $V$. This is the restriction of the $\ell^1$ metric to the $0$-skeleton.

\subsection{Normal cube paths}
Given a cube $C \in X$,  we denote by $\mathrm{St}(C)$ the union of all cubes which contain $C$ as a subface.
\begin{defn}[\cite{niblo1998geometry}]
Let  $\{C_i\}_{i=0}^n$ be a sequence of cubes such that each cube has dimension at least 1, and $C_{i-1} \cap C_i$ consists of a single point, denoted by $v_i$.
\begin{itemize}
  \item Call $\{C_i\}_{i=0}^n$ a \emph{cube path}, if $C_i$ is the (unique) cube of minimal dimension containing $v_i$ and $v_{i+1}$, i.e. $v_i$ and $v_{i+1}$ are diagonally opposite vertices of $C_i$. Define $v_0$ to be the vertex of $C_0$ diagonally opposite to $v_1$, and $v_{n+1}$ to be the vertex of $C_n$ diagonally opposite to $v_n$. The so-defined vertices $\{v_i\}_{i=0}^{n+1}$ are called the vertices of the cube path, and we say the cube path is from $v_0$ to $v_{n+1}$.
  \item The \emph{length} of a cube path is the number of the cubes in the sequence.
  \item A cube path is called \emph{normal} if $C_i \cap \mathrm{St}(C_{i-1})=v_i$.
\end{itemize}
\end{defn}

Normal cube paths in CAT(0) cube complexes behave like geodesics in trees. More precisely, in~\cite{niblo1998geometry},  the existence and uniqueness of normal cube paths connecting any pair of vertices is established. See also~\cite{reeves1995biautomatic}.

\begin{prop}[\cite{niblo1998geometry}]
For any two vertices $x,y \in V$, there exists a unique normal cube path from $x$ to $y$. (Note that the order is important here since in general normal cube paths are not reversible).
\end{prop}

\begin{prop}[\cite{niblo1998geometry}]\label{normal cube path intersection property}
The intersection of a normal cube path and a hyperplane is connected. In other words, a normal cube path crosses a hyperplane at most once.
\end{prop}

\begin{prop}[\cite{niblo1998geometry},~\cite{campbell2005hilbert}]\label{fellow-traveller}
Let $\{C_i\}_{i=0}^n$ and $\{D_j\}_{j=0}^m$ be two normal cube paths in $X$, and let $\{v_i\}_{i=0}^{n+1}$ and $\{w_j\}_{j=0}^{m+1}$ be the vertices of these normal cube paths. If $d(v_0,w_0) \leqslant 1$ and $d(v_{n+1},w_{m+1}) \leqslant 1$, then for all $k$,  we have $d(v_k,w_k) \leqslant 1$.
\end{prop}

We omit the proofs for the above three propositions, the readers can find them in the original paper. However, let us recall the construction of the normal cube path from $x$ to $y$ as follows: consider all the hyperplanes separating $x$ from $y$ and adjacent to $x$. The key fact is that these hyperplanes all cross a unique cube adjacent to $x$ lying in the interval from $x$ to $y$. This cube is defined to be the first cube on the normal cube path; then one proceeds inductively to construct the required normal cube path.

We will also need the following lemma, abstracted from \cite{niblo1998geometry}.

\begin{lem}\label{normal cube path main lemma}
Let $\{C_i\}_{i=0}^n$ be the normal cube path, and $h$ be a hyperplane. If $h \pitchfork C_i$, then $\exists$ a hyperplane $k$, such that $k\pitchfork C_{i-1}$ and $h$ does not intersect with $k$.
\end{lem}

\begin{proof}
Otherwise, $\forall k \pitchfork C_{i-1}$, we have $h \pitchfork k$. Now by Lemma 2.15 in~\cite{niblo1998geometry}, we know that there exists a cube $C \in X$, such that all such $k \pitchfork C$ and $h \pitchfork C$, and $C_{i-1}$ is a face of $C$. Moreover, $C$ contains an edge $e$ of $C_i$ since $h \pitchfork C$. So $\mathrm{St}(C_{i-1}) \cap C_i$ contains $e$, which is a contradiction to the definition of normal cube path.
\end{proof}

Now for any two vertices of $X$, we consider all the hyperplanes separating them, with a partial order by inclusion. More explicitly, for any $x,y\in V$, let $H(x,y)$ be the set of hyperplanes separating $x$ and $y$. For any $h\in H(x,y)$, let $h^-$ be the halfspace containing $x$. Define $h \leqslant k$ if $h^- \subseteq k^-$. Note that the definition depends on the vertices we choose, and we may change them under some circumstances, but still write $h^-$ for abbreviation. To avoid ambiguity, we point out the vertices if necessary. We write $h<k$ to mean a strict containment $h^- \subsetneq k^-$.
\begin{lem}\label{normal contain}
For any $h,k \in H(x,y)$, $h$ and $k$ do not intersect if and only if $h \leqslant k$ or $k \leqslant h$.
\end{lem}

\begin{proof}
We only need to show the necessity.
Let $\{C_i\}_{i=0}^n$ be the normal cube path from $x$ to $y$, and assume $h \pitchfork C_i$, $k \pitchfork C_j$. Since $h$ and $k$ do not intersect, $i \neq j$. Assume $i<j$.
Obviously, $x \in h^{-} \cap k^{-}$, and $y \in h^+ \cap k^+$. Since $h \pitchfork C_i$ and $k \pitchfork C_j$, by Proposition \ref{normal cube path intersection property}, $v_{i+1} \in h^+ \cap k^-$. Since $h$ does not intersect with $k$, we have $h^- \cap k^+ = \emptyset$, which implies $h^{-} \subseteq k^{-}$.
\end{proof}

Combining the above two lemmas, we have the following result on the existence of chains in $H(x,y)$.

\begin{prop}\label{normal chain}
Let $\{C_i\}_{i=0}^n$ be the normal cube path from $x$ to $y$, and $h$ be a hyperplane such that $h \pitchfork C_l$. Then there exists a chain of hyperplanes $h_0 < h_1 < \cdots < h_{l-1}, h_l=h$ such that $h_i \pitchfork C_i$.
\end{prop}

\begin{proof}
By Lemma~\ref{normal cube path main lemma}, there exists a hyperplane $k$, such that $k\pitchfork C_{l-1}$ and $h$ does not intersect with $k$. Define $h_{l-1}=k$. Inductively, we can define a sequence of hyperplanes as required. Then the conclusion follows by Lemma~\ref{normal contain}.
\end{proof}

Finally, we give a lemma used in the proof of the consistency part of our main theorem.
\begin{lem}\label{consistency lemma}
Let $x_0,x,y\in V$ with $[x_0,y] \subseteq [x_0,x]$, and let $x',y'$ be the $n-$th vertex on the normal cube path from $x_0$ to $x$, and to $y$. If $x' \neq y'$, then $x' \notin [x_0,y]$.
\end{lem}

\begin{proof}
Otherwise, $x' \in [x_0,y]$. By the construction of the normal cube path, we know $x'$ is also the $n-$th vertex on the normal cube path from $x_0$ to $y$, since $y\in [x_0,x]$. In other words, $x'=y'$, which is a contradiction to the assumption.
\end{proof}

\subsection{Normal metric}
We define a new metric on $V=X^{(0)}$ using normal cube paths~\cite{niblo1998geometry, reeves1995biautomatic}.

\begin{defn}
For any $x,y\in V$, define $d_{nor}(x,y)$ to be the length of the normal cube path from $x$ to $y$. We call $d_{nor}$ the \emph{normal metric} on $V$.
\end{defn}

One needs to verify that $d_{nor}$ is indeed a metric. It is easy to see that $d_{nor}(x,y) \geqslant 0$, and $d_{nor}(x,y)=0$ if and only if $x=y$. Note that the normal cube path from $x$ to $y$ is not the one from $y$ to $x$ in general, so the symmetric relation is not that obvious. In order to show the symmetric relation and the triangle inequality, we give the following characterization.

\begin{lem}\label{char nor dist}
For $x,y \in V$, let $<$ be the relation defined as above. Then
$$d_{nor}(x,y)=\sup\{m+1: h_0 <h_1 <\cdots <h_m, h_i \in H(x,y)\}.$$
\end{lem}

\begin{proof}
Suppose $\{C_i\}_{i=0}^n$ is the normal cube path from $x$ to $y$, so $d_{nor}(x,y)=n+1$. Denote the right hand side of the equality in the lemma by $n'$. Now for any chain $h_0 < h_1 < \cdots < h_m$ in $H(x,y)$, by Proposition \ref{normal cube path intersection property}, $h_i$ intersects with just one cube, denoted by $C_{k(i)}$. Obviously, if $h,k \pitchfork C_i$, then $h\pitchfork k$. So $k(i) \neq k(j)$ if $i\neq j$, which implies $m \leqslant n$, so $n' \leqslant n$.

On the other hand, for any $h \pitchfork C_n$, by Proposition \ref{normal chain}, we have a chain of hyperplanes $h_0<h_1<\cdots<h_{n-1}<h_n=h$, such that $h_i \pitchfork C_i$, which implies $n \leqslant n'$.
\end{proof}

\begin{prop}
$d_{nor}$ is indeed a metric on $V$.
\end{prop}

\begin{proof}
By Lemma \ref{char nor dist},  $H(x,y)=H(y,x)$, and as posets they carry opposite orders. One can thus deduce $d_{nor}(x,y)=d_{nor}(y,x)$. For $x,y,z\in V$, $H(x,y) \triangle H(y,z)=H(x,z)$, where $\triangle$ is the symmetric difference operation. The inclusions of $H(x,y)\cap H(x,z)$ into $H(x,y)$ and $H(y,z)\cap H(x,z)$ into $H(y,z)$ are both order preserving, and therefore, by Lemma \ref{char nor dist}, we have $d_{nor}(x,z) \leqslant d_{nor}(x,y)+d_{nor}(y,z)$.
\end{proof}

\subsection{Normal balls and normal spheres}
Recall that for any two points $x,y$ in $V=X^{(0)}$, the interval between them is
$$[x,y]=\{z \in V: d(x,y)=d(x,z)+d(z,y)\}.$$
In other words, $[x,y]$ is the set of vertices on the union of all the edge geodesics from $x$ to $y$. A subset $Y \subseteq V$ is called \emph{convex}, if for any $x,y \in Y$, $[x,y] \subseteq Y$.

Now let $B(x,n)$ be the closed ball in the edge metric with centre $x \in V$ and radius $n$. Generally, $B(x,n)$ is not convex (for example, take $X=\mathbb{Z}^2$). However, as we will see, for the normal metric balls are convex. More precisely, we define the \emph{normal ball} with centre $x \in V$ and radius $n$ to be
$$B_{nor}(x,n)=\{y \in V: d_{nor}(x,y) \leqslant n\}$$
and the \emph{normal sphere}
with centre $x \in V$ and radius $n$ to be
$$S_{nor}(x,n)=\{y \in V: d_{nor}(x,y)= n\}.$$

\begin{lem}\label{conv}
$B_{nor}(x,n)$ is convex for all $x\in V$ and $n \in \mathbb{N}$.
\end{lem}

\begin{proof}
Given $z,w \in B_{nor}(x,n)$, and a geodesic $\gamma$ from $z$ to $w$, if $\gamma \nsubseteq B_{nor}(x,n)$, we can assume $u$ is the first vertex on $\gamma$ which is not in $B_{nor}(x,n)$, which implies $d_{nor}(x_0,u)=n+1$. Let $z'$ be the vertex preceding $u$ on $\gamma$, so $d_{nor}(x_0,z')=n$ (since $d_{nor}(z',u)=1$). Since $d(z',u)=1$, there exists a unique hyperplane $h$ separating $z'$ from $u$, so $H(x_0,u)=H(x_0,z') \sqcup \{h\}$. Now
according to Lemma \ref{char nor dist}, there exists a chain $h_0 < \cdots < h_{n-1} < h$ in $H(x_0,u)$ with $h_i \in H(x_0,z')$. Since every geodesic intersects with any hyperplane at most once (see for example~\cite{sageev1995ends}), $w \in h^+$, which implies $h_0 < \cdots < h_{n-1} < h$ is also a chain in $H(x_0,w)$. This is a contradiction to $d_{nor}(x_0,w) \leqslant n,$ by Lemma \ref{char nor dist}.
\end{proof}

Since the intersection of two convex sets is still convex, we have the following corollary.
\begin{cor}\label{conv cor}
For any $x \in V$ and $n \in \mathbb{N}$, the set $[x_0,x] \cap B_{nor}(x_0,n)$ is convex.
\end{cor}

It is well known that for a convex subset $Y$ in a CAT(0) cube complex and a point $v \notin Y$, there is a unique point in $Y$ which is closest to $v$ (see, for example,~\cite{BH99}). This statement is true both for the intrinsic CAT(0) metric on the cube complex and the edge metric on the vertex set,  and we have a similar statement for the normal distance:

\begin{prop}\label{def for v}
There exists a unique point $v\in [x_0,x] \cap B_{nor}(x_0,n)$ such that
$$[x_0,x] \cap B_{nor}(x_0,n) \subseteq [x_0,v].$$

The point $v$  is characterized by:

$$d(x_0,v)=\max\{d(x_0,v'): v'\in [x_0,x] \cap B_{nor}(x_0,n)\}.$$

Furthermore, if $d_{nor}(x_0,x) \geqslant n$, then $v\in [x_0,x] \cap S_{nor}(x_0,n)$, which implies that $v$ is also the unique point in $[x_0,x] \cap S_{nor}(x_0,n)$ such that
$$d(x_0,v)=\max\{d(x_0,v'): v'\in [x_0,x] \cap S_{nor}(x_0,n)\}.$$
\end{prop}

\begin{proof}
If there exist $z\neq w \in [x_0,x] \cap B_{nor}(x_0,n)$ such that $d(x_0,z)=d(x_0,w)$ attains the maximum, consider the median $m=\mu(z,w,x)$. By Corollary \ref{conv cor}, $m\in [z,w] \subseteq [x_0,x] \cap B_{nor}(x_0,n)$, so $d(m,x_0)=d(z,x_0)=d(w,x_0)$. While $m\in [z,x] \cap [w,x]$, so $m=z=w$, which is a contradiction.

By Corollary \ref{conv cor} $[x_0,v]\subseteq [x_0,x] \cap B_{nor}(x_0,n)$. Conversely, for any $u\in [x_0,x] \cap B_{nor}(x_0,n)$, let $m=\mu(u,v,x) \in [u,v]$. By Corollary \ref{conv cor}, $m\in [x_0,x] \cap B_{nor}(x_0,n)$. While $m \in [v,x]$, so $d(m,x_0) \geqslant d(v,x_0)$, which implies $m=v$ by the choice of $v$, i.e. $\mu(u,v,x)=v$, so $v \in [u,x]$. Now by Lemma \ref{basic median}, $u \in [x_0,v]$.

Now for $x,n$ satisfying $d_{nor}(x_0,x) \geqslant n$, if $v \in [x_0,x] \cap B_{nor}(x_0,n-1)$, take a geodesic $\gamma$ from $v$ to $x$, and let $v=y_0,y_1,\ldots,y_k=x$ be the vertices on $\gamma$. Since $d_{nor}(x_0,x)\geqslant n$, $x\neq v$, which implies $k>0$. Now for $y_1$, since $y_1 \in [v,x]$, $d(x_0,v)<d(x_0,y_1)$. By the definition of $v$, we know $y_1 \notin [x_0,x] \cap B_{nor}(x_0,n)$, so $y_1 \notin B_{nor}(x_0,n)$. However, since $d(v,y_1)=d_{nor}(v,y_1)=1$, we have
$$d_{nor}(x_0,y_1) \leqslant d_{nor}(x_0,v) + d_{nor}(v,y_1) \leqslant n,$$
which is a contradiction.
\end{proof}

To use the above proposition more flexibly, we give another characterization of $v$, which can also be viewed as an alternative definition of $v$. In the rest of this subsection, we fix  $x\in V$ and $n\in \mathbb{N}$ with $d_{nor}(x_0,x) \geqslant n$.

\begin{prop}\label{char for v}
Let $\{C_i\}_{i=0}^N$ be the normal cube path from $x_0$ to $x$, and $\tilde{v}=v_n$ be the n-th vertex on this normal cube path. Then $\tilde{v}=v$, which is provided by Proposition~\ref{def for v}.
\end{prop}

To prove this result, let us focus on subsets in $H(x_0,x)$. Recall that $H(x_0,x)$ is endowed with the relation $\leqslant$, as defined prior to Lemma \ref{normal contain}.
\begin{defn}
A subset $A \subseteq H(x_0,x)$ is called \emph{closed} (under $<$), if $\forall h \in A$ and $k<h$, $k\in A$.
\end{defn}

\begin{lem}\label{max}
Let $\tilde{v}$ be the $n$-th vertex on the normal cube path from $x_0$ to $x$, then $H(x_0,\tilde{v})$ is maximal in the following sense: for any closed $A \subseteq H(x_0,x)$ which contains chains only with lengths at most $n$, $A\subseteq H(x_0,\tilde{v})$.
\end{lem}

\begin{proof}
We proceed by induction on $n$. Suppose that the lemma holds for $n-1$, and let $v'$ be the $(n-1)-$th vertex on the normal cube path from $x_0$ to $x$. Given a closed $A\subseteq H(x_0,x)$ containing chains only with lengths at most $n$, and a maximal chain $h_0<h_1<\cdots <h_m$ in $A$. If $m \leqslant n-2$, then the closed set $\{h\in A: h\leqslant h_m\}$ contains chains only with lengths at most $n-1$; by induction, it is contained in $H(x_0,v')\subseteq H(x_0,\tilde{v})$. Now for $m=n-1$: similarly, $\{h\in A: h\leqslant h_{n-2}\} \subseteq H(x_0,v')$, which implies $h_i \pitchfork C_i$ for $i=0,1,\ldots,n-2$. So $h_{n-1}\pitchfork C_k$ for some $k \geqslant n-1$. If $k \neq n-1$, by Proposition~\ref{normal chain} and the closeness of $A$, we get a chain in $A$ with length greater than $n$, which is a contradiction. So $h_{n-1}\pitchfork C_{n-1}$, i.e. $h_{n-1} \in H(x_0,\tilde{v})$.
\end{proof}

\begin{proof}[Proof of Proposition \ref{char for v}]
By Proposition \ref{def for v}, $\tilde{v} \in [x_0,x] \cap B_{nor}(x_0,n) \subseteq [x_0,v]$, which implies $H(x_0,\tilde{v}) \subseteq H(x_0,v)$. However, $H(x_0,v)$ is closed and contains chains only with lengths at most $n$ according to Lemma \ref{char nor dist}, so $H(x_0,v) \subseteq H(x_0,\tilde{v})$ by Lemma~\ref{max}, which implies $H(x_0,v)=H(x_0,\tilde{v})$. So $H(v,\tilde{v})=H(x_0,v) \triangle H(x_0,\tilde{v}) = \emptyset$, which implies $v=\tilde{v}$.
\end{proof}

Finally, we characterize those points in  $[x_0,x]$ which lie in the intersection $[x_0,x] \cap S_{nor}(x_0,n)$. This will be used in the next subsection to decompose $[x_0,x] \cap S_{nor}(x_0,n)$ into a union of intervals.

Let $C_{n-1}$ be the $n$-th cube on the normal cube path from $x_0$ to $x$, and $v=\tilde{v}$ is the $n$-th vertex on the cube path as above. Let $H_n$ be the set of all hyperplanes intersecting with $C_{n-1}$.

\begin{prop}\label{dec prel}
For $w\in [x_0,x]$, the following are equivalent:
\begin{enumerate}[1)]
  \item $w \in [x_0,x] \cap S_{nor}(x_0,n)$;
  \item $\exists h\in H_n$, such that $h$ crosses the last cube on the normal cube path from $x_0$ to $w$;
  \item $\exists h \in H_n$, such that $h$ separates $w$ from $x_0$ and $w\in[x_0,v]$.
\end{enumerate}
\end{prop}

\begin{proof}
~\\[-0.5cm]
\begin{itemize}
  \item $1)\Rightarrow 3)$: By Proposition \ref{def for v}, $w\in [x_0,v]$. Since $d_{nor}(x_0,w)=n$, by Lemma~\ref{char nor dist}, the maximum length of chains in $H(x_0,w)$ is $n$. Take such a chain $h_0<h_1<\cdots <h_{n-1}$ in $H(x_0,w) \subseteq H(x_0,v)$. Obviously, $h_i$ intersects with different cubes, which implies $h_i \pitchfork C_i$. So $h_{n-1} \in H_n$, and it separates $w$ from $x_0$.
  \item $3)\Rightarrow 2)$: Since $h$ separates $w$ from $x_0$, $h$ must cross some cube $C$ on the normal cube path from $x_0$ to $w$. Since $h\in H_n$, we know there is a chain $h_0<h_1<\cdots<h_{n-1}=h$ in $H(x_0,v)$, which is also a chain in $H(x_0,w)$. So $h$ cannot cross the first $n-1$ cubes of the normal cube path from $x_0$ to $w$. If $h$ does not cross the last cube, then $d_{nor}(x_0,w)>n$. However, $w\in [x_0,v]$ implies $H(x_0,w) \subseteq H(x_0,v)$, by Lemma \ref{char nor dist}, $d_{nor}(x_0,v) >n$, which is a contradiction.
  \item $2)\Rightarrow 1)$: This is immediate, by Lemma \ref{char nor dist}.
\end{itemize}
\end{proof}

We have another description for $H_n$, which is implied by Proposition~\ref{normal chain} directly.
\begin{lem}\label{char for H_n}
For $h\in H(x_0,x)$, $h\in H_n$ if and only if the maximal length of chains in $\{k\in H(x_0,x): k\leqslant h\}$ is $n$.
\end{lem}

\subsection{Decomposition of $[x_0,x] \cap S_{nor}(x_0,n)$}
We want to decompose the set $[x_0,x] \cap S_{nor}(x_0,n)$ so that we can proceed by the induction on dimension in the proof of Theorem~\ref{thm a}.

Throughout this subsection, we fix $x\in V$ and $n\in \mathbb{N}$ with $d_{nor}(x_0,x) \geqslant n$, and let $v$ be as defined in Proposition \ref{def for v}. At the end of the preceding subsection, we have defined $H_n$ to be the set of all hyperplanes intersecting with $C_{n-1}$, where $\{C_i\}$ is the normal cube path from $x_0$ to $x$.

Now we decompose $[x_0,x] \cap S_{nor}(x_0,n)$ into a union of intervals with dimensions lower than $[x_0,x]$, and the number of these intervals can be controlled by the dimension of $[x_0,x]$. This will make it possible to do induction on the dimension.

\begin{defn}
For $h\in H_n$,  we define
$$F_h=\{w\in [x_0,x] \cap S_{nor}(x_0,n): h \mbox{~separates~} w \mbox{~from~} x_0\}.$$
\end{defn}

By Proposition \ref{dec prel}, we immediately obtain the following two lemmas.
\begin{lem}\label{def for Fh}
\begin{eqnarray*}
        F_h & = & \{w\in [x_0,x]: h \mbox{~crosses~the~last~cube~on~the~normal~cube~path~from~}x_0 \mbox{~to~}w\} \\
            & = & \{w\in[x_0,v]: h \mbox{~separates~} w \mbox{~from~} x_0\}.
\end{eqnarray*}
\end{lem}

\begin{lem}\label{dec}
$[x_0,x] \cap S_{nor}(x_0,n)=\bigcup_{h\in H_n}F_h$.
\end{lem}

By definition, we know
$$F_h = [x_0,x] \cap B_{nor}(x_0,n) \cap \{v':h \mbox{~separates~} v' \mbox{~from~} x_0\},$$
which implies that $F_h$ is convex. Moreover, we will show that $F_h$ is actually an interval.

\begin{lem}\label{label}
Let $x_h\in F_h$ be the  point minimising $d(x_0,x_h)$. Then $F_h=[x_h,v]$.
\end{lem}

\begin{proof}
Since $F_h$ is convex and $x_h, v\in F_h$, so $[x_h,v] \subseteq F_h$. On the other hand, $\forall z\in F_h$, let $m=\mu(x_0,z,x_h)$. So, $m\in F_h$ and $d(x_0,m) \leqslant d(x_0,x_h)$. By the choice of $x_h$, we know that $d(x_0,m)=d(x_0,x_h)$, which implies $m=x_h$, so $x_h \in [x_0,z]$. By Proposition \ref{def for v}, $x_h,z\in [x_0,v]$. Thus, by Lemma \ref{basic median}, $z\in [x_h,v]$.
\end{proof}

\begin{prop}\label{dec final}
$[x_0,x] \cap S_{nor}(x_0,n)=\bigcup_{h\in H_n}[x_h,v]$, and $\dim[x_h,v] < \dim[x_0,x]$.
\end{prop}

\begin{proof}
We only need to show $\dim[x_h,v] < \dim[x_0,x]$. For any hyperplane $k$ crossing $[x_h,v]$, by Proposition \ref{dec prel}, $k\pitchfork h$. So $\dim[x_h,v]< \dim[x_0,x]$.
\end{proof}

Now we give another characterization for $x_h$, which is useful in the proof of the consistency condition of Theorem~\ref{thm a}.

\begin{lem}\label{char for x_h}
Let $x_h$ be the closest point to $x_0$ on $F_h$, then $x_h$ is the unique point in $B_{nor}(x_0,n)$ such that $h$ separates $x_0$ from $x_h$, and for any hyperplane $k \pitchfork h$, $k$ does not separate $x_h$ from $x_0$.
\end{lem}

\begin{proof}
Since $x_h \in F_h$, we have $x_h\in [x_0,x] \cap B_{nor}(x_0,n)$ and $h$ separates $x_0$ from $x_h$. Now for any hyperplane $k \pitchfork h$, if $k$ separates $x_h$ from $x_0$, we have $x_h \in h^+ \cap k^+$ and $x_0 \in h^- \cap k^-$. Choose $\tilde{x}_h \in [x_0,x_h]$ such that $\tilde{x}_h \in h^+ \cap k^-$. Since $k$ does not separate $\tilde{x}_h$ from $x_0$, so $d(x_0,\tilde{x}_h) < d(x_0,x_h)$. However, by Lemma \ref{def for Fh}, $\tilde{x}_h \in F_h$, which is a contradiction.

It remains to show that $x_h$ is the unique point satisfying these conditions. Otherwise, let $\hat{x}_h$ be another point satisfying the hypothesis in the lemma and $\hat{x}_h \neq x_h$. Let $k$ be a hyperplane separating $\hat{x}_h$ from $x_h$, and assume $x_h \in k^-$. Obviously, $k \neq h$. If $k\pitchfork h$, by hypothesis, $k$ does not separate $x_h$ from $x_0$, as well as $\hat{x}_h$ from $x_0$, which is a contradiction since $k$ separates $\hat{x}_h$ from $x_h$. So $k$ does not cross $h$, which implies $h^- \subsetneq k^-$ by Lemma \ref{normal contain}. However, $\hat{x}_h \in k^+$, so by Lemma \ref{char nor dist}, $d_{nor}(x_0,x_h) < d_{nor}(x_0,\hat{x}_h)$. This is a contradiction since $d_{nor}(x_0,x_h)\geqslant n$ as $h$ separates $x_0$ from $x_h$.
\end{proof}

\subsection{\u{S}pakula and Wright's Construction.}
We conclude this section with a recent application of normal cube paths, which were invoked by \u{S}pakula and Wright, ~\cite{spakula2016coarse}, in order to provide a new proof that finite dimensional CAT(0) cube complexes  have Yu's Property A. 
The  key to their proof was the construction of a family of maps $h_l$ with the property that for any interval and any neighbourhood of an endpoint of the interval  the maps push that neighbourhood into the interval itself. These maps were defined in terms of the normal cube paths as follows:
 
 \begin{defn}[The $h$ maps]
Given $l\in \mathbb{N}$, we define $h_l\colon X \rightarrow X$ as follows. For $x\in X$, let $h_l(x)$ be the $3l-$th vertex on the normal cube path from $x$ to $x_0$ if $d_{nor}(x,x_0) \geqslant 3l$; and let  it be $x_0$ if $d_{nor}(x,x_0) < 3l$.
\end{defn}

\begin{lem}[\cite{spakula2016coarse}]\label{SW}
Let $h_l$ be defined as above and $y \in B(x,3l)$. Then $h_l(y) \in [x_0,x]$.
\end{lem}

\begin{proof}
We only need to show that every halfspace containing $x$ and $x_0$ contains also $z=h_l(y)$. For any hyperplane $h$ such that one of the associated halfspaces, say $h^+$, contains $x$ and $x_0$, either $y\in h^+$ or $y\in h^-$. In the former case, $z\in h^+$, so we only need to check the case that $h$ separates $x,x_0$ from $y$.

Denote by $C_0,C_1,\ldots, C_m$ the normal cube path from $y$ to $x_0$, and denote by $y=v_0,v_1,\ldots,v_m=x_0$ the vertices on this cube path. We shall argue that any hyperplane separating $y$ from $x,x_0$ is ``used" within the first $d(x,y)$ steps on the cube path. Suppose that the cube $C_i$ does not cross any hyperplane $h$ with $h$ separating $y$ from $x,x_0$. Hence every hyperplane $k \pitchfork C_i$ separates $y,x$ from $x_0,v_{i+1}$. If there was a hyperplane $l$ separating $y$ from $x,x_0$ before $C_i$, then necessarily $l$ separates $y,v_{i+1}$ from $x,x_0$, hence $l$ crosses all the hyperplanes $k$ crossing $C_i$. This contradicts the maximality of this step on the normal cube path. Thus, there is no such $l$, and so all the hyperplanes $h$ separating $y$ from $x,x_0$ must be crossed within the first $d(x,y)$ steps.

Since $z$ is the $3l-$th vertex on the cube path and $d(x,y) \leqslant 3l$, all the hyperplanes $h$ separating $y$ from $x,x_0$ must have been crossed before $z$. Thus, any such $h$ actually also separates $y$ from $x,x_0,z$.
\end{proof}

We will use the remarkable properties of the $h$ maps to construct the $S$ sets defined in our characterization of finite asymptotic dimension in the next section.

\section{Finite dimensional CAT(0) cube complexes}

Throughout this section, we fix a CAT(0) cube complex $X$ of finite dimension $\rank$ and equipped with a basepoint $x_0\in X$. We will make use of  the characterization obtained in Corollary~\ref{char for asdim} in order to prove Theorem~\ref{thm a}.

\subsection{Constructing the sets $S(x,k,l)$.}
By Corollary \ref{char for asdim}, in order to prove $X$ has finite asymptotic dimension, we need to find a constant $N \in \mathbb{N}$ such that $\forall l \in \mathbb{N}$, $\forall k=1,2,\ldots, 3l$, $\forall x \in X$, we can assign a subset $S(x,k,l) \subseteq X$, satisfying:
\begin{enumerate}[i)]
  \item $\forall l \in \mathbb{N}$, $\exists S_l >0$, such that $S(x,k,l) \subseteq B(x,S_l)$ for all $k=1,\ldots, 3l$;
  \item $\forall l \in \mathbb{N}$, $\forall k,k'$ with $1 \leqslant k \leqslant k' \leqslant 3l$, $\forall x \in X$,
        $S(x,k,l) \subseteq S(x,k',l)$;
  \item $\forall x,y \in X$ with $d(x,y)=1$,
        $S(y,k,l) \subseteq S(x,k+1,l)$ for all $k=1,2,\ldots, 3l-1$;
  \item $\forall l \in \mathbb{N}$, $\sharp S(x,2l,l) \leqslant N+1$.
\end{enumerate}

Now for $l\in \mathbb{N}$, $k=1,2,\ldots,3l$, and $x\in X$, we define
$$\widetilde{S}(x,k,l)=h_l(B(x,k)).$$
It is easy to show that $\{\widetilde{S}(x,k,l)\}$ satisfies i) to iii), but it does not satisfy iv) above, so we need some modification. Intuitively, we construct $S(x,k,l)$ as a uniformly separated net in $\widetilde{S}(x,k,l)$. To be more precise, we require the following lemma.

\begin{lem}\label{main lemma}
There exist two constants $N,K$ only depending on the dimension $\rank$, such that $\forall l \in \mathbb{N}$, $\forall x \in V$, there are subsets $C_x \subseteq [x_0,x]$, and  maps $p_x\colon [x_0,x] \rightarrow \mathcal{P}(C_x)$, where $\mathcal{P}(C_x)$ denotes the power set of $C_x$, satisfying:
\begin{itemize}
  \item If $d(x,y)=1$ and $y \in [x_0,x]$, then $C_x \cap [x_0,y]=C_y$, and $p_x\vert_{[x_0,y]}=p_y$;
  \item For $z \in [x_0,x]$ and $w \in p_x(z)$, we have $d(z,w) \leqslant Kl$;
  \item $\forall z\in [x_0,x]$, $\sharp \big(B(z,Ml) \cap C_x\big) \leqslant N$, where $M=3 \rank + 3 + K$.
\end{itemize}
\end{lem}

We postpone the proof of the above lemma and first show how to use it to construct $S(x,k,l)$ (and, hence, to conclude the proof of Theorem~\ref{thm a}).

\begin{proof}[Proof of Theorem \ref{thm a}.]
Let $N,K$ be the constants in Lemma \ref{main lemma}. $\forall l \in \mathbb{N}$, $\forall k=1,2,\ldots, 3l$, $\forall x \in X$, let $\widetilde{S}(x,k,l)=h_l(B(x,k))$ be as above, and by Lemma \ref{SW}, we know $\widetilde{S}(x,k,l) \subseteq [x_0,x]$. Now we define
$$S(x,k,l)=\bigcup p_x(\widetilde{S}(x,k,l)),$$
and the only thing left to complete the proof is to verify the conditions in Corollary \ref{char for asdim}.
\begin{enumerate}[i)]
  \item By the definition of $h_l$, we know $\forall y \in B(x,k)$, $d(y,h_l(y)) \leqslant 3\rank l$. So for any $z \in \widetilde{S}(x,k,l)$, $d(z,x)\leqslant (3\rank+3)l$. For such $z$ and any $w \in p_x(z)$, by Lemma \ref{main lemma}, we know $d(z,w) \leqslant Kl$, which implies:
      $$S(x,k,l) \subseteq B(x,(3\rank+3+K)l) = B(x,Ml).$$
  \item $\forall l \in \mathbb{N}$, $\forall k,k'$ with $1 \leqslant k \leqslant k' \leqslant 3l$, $\forall x \in X$, we have $\widetilde{S}(x,k,l) \subseteq \widetilde{S}(x,k',l)$. Now immediately by the definition, $S(x,k,l) \subseteq S(x,k',l)$.
  \item $\forall x,y \in X$ with $d(x,y)=1$, by Lemma \ref{weakly modular}, $y\in [x_0,x]$ or $x\in [x_0,y]$. Assume the former. Let $k=1,2,\ldots, 3l-1$.
      Obviously, $\widetilde{S}(y,k,l) \subseteq \widetilde{S}(x,k+1,l)$, so we have
      \begin{eqnarray*}
        S(y,k,l) & = & \bigcup p_y(\widetilde{S}(y,k,l))  =  \bigcup p_x|_{[x_0,y]}(\widetilde{S}(y,k,l))  =  \bigcup p_x(\widetilde{S}(y,k,l)) \\
                 & \subseteq & \bigcup p_x(\widetilde{S}(x,k+1,l))  =  S(x,k+1,l).
      \end{eqnarray*}
      Here we use the first part of Lemma \ref{main lemma} in the second equation.
      On the other hand, $\widetilde{S}(x,k,l) \subseteq \widetilde{S}(y,k+1,l)$, so we have
      \begin{eqnarray*}
        S(x,k,l) & = & \bigcup p_x(\widetilde{S}(x,k,l))  \subseteq  \bigcup p_x(\widetilde{S}(y,k+1,l)) \\
                 & = & \bigcup p_x|_{[x_0,y]}(\widetilde{S}(y,k+1,l))  =  \bigcup p_y(\widetilde{S}(y,k+1,l)) = S(y,k+1,l).
      \end{eqnarray*}
      Here we use the first part of Lemma \ref{main lemma} in the fourth equality.

  \item By i), we know that $S(x,k,l) \subseteq B(x,Ml)$ for all $k=1,2,\ldots,3l$. Hence, by definition, $S(x,k,l) \subseteq B(x,Ml) \cap C_x$. Now by the third part of Lemma \ref{main lemma}, we have $\sharp S(x,k,l) \leqslant N$.
\end{enumerate}
\end{proof}

The last thing is to prove Lemma \ref{main lemma}. We use the analysis in Section~\ref{sec:normal} to construct $C_x$ and $p_x$ inductively. Recall that in Section~\ref{sec:normal} (Proposition~\ref{dec final}), for any $l,n\in \mathbb{N}$, and any $x\in X$, we have
$$[x_0,x] \cap S_{nor}(x_0,nl)=\bigcup_{h\in H_{nl}}[x_h,v],$$
with $\sharp H_{nl} \leqslant \rank$ and $\dim[x_h,v] < \dim[x_0,x]$. In order to carry out induction on the dimension of $[x_0,x]$, we require a stronger version of Lemma~\ref{main lemma}, which is more flexible on the choice of endpoints of intervals. More explicitly, we have
\begin{lem}\label{strong main lemma}
There exist two constants $N,K$ only depending on the dimension $\rank$, such that $\forall l \in \mathbb{N}$, $\forall \bar{x},x \in V$, $\exists C_{\bar{x},x} \subseteq [\bar{x},x]$, and a map $p_{\bar{x},x}\colon[\bar{x},x] \rightarrow \mathcal{P}(C_{\bar{x},x})$ satisfying:
\begin{itemize}
  \item If $d(x,y)=1$ and $y \in [\bar{x},x]$, then $C_{\bar{x},x} \cap [\bar{x},y]=C_{\bar{x},y}$, and $p_{\bar{x},x}|_{[\bar{x},y]}=p_{\bar{x},y}$;
  \item For $z \in [\bar{x},x]$ and $w \in p_{\bar{x},x}(z)$, we have $d(z,w) \leqslant Kl$;
  \item $\forall z\in [\bar{x},x]$, $\sharp \big(B(z,Ml) \cap C_{\bar{x},x}\big) \leqslant N$, where $M=3 \rank + 3 + K$.
\end{itemize}
\end{lem}
It is obvious that Lemma \ref{main lemma} is implied by Lemma \ref{strong main lemma} (one just needs to take $\bar{x}=x_0$). Now we prove Lemma \ref{strong main lemma}.

\begin{proof}[Proof of Lemma \ref{strong main lemma}]
Fix an $l \in \mathbb{N}$. We will carry out induction on $\dim[\bar{x},x]$.

Given any $\bar{x},x \in V$ with $\dim[\bar{x},x]=1$, we define
$$C_{\bar{x},x}=\{y\in [\bar{x},x]: d_{nor}(\bar{x},y) \in l\mathbb{N}\},$$
where $l\mathbb{N}=\{0,l,2l,3l,\ldots\}$. Since $\dim[\bar{x},x]=1$, $[\bar{x},x]$ is indeed isometric to an interval in $\mathbb{R}$. We define $p_{\bar{x},x}\colon [\bar{x},x] \rightarrow \mathcal{P}(C_{\bar{x},x})$ as follows: for any $y\in [\bar{x},x]$, $p_{\bar{x},x}(y)$ consists of a single point which is at distance $l\lfloor d_{nor}(\bar{x},y)/l \rfloor$ from $\bar{x}$ in $[\bar{x},y]$, where $\lfloor \cdot \rfloor$ is the function of taking integer part. Now it is obvious that
\begin{itemize}
  \item If $d(x,y)=1$ and $y \in [\bar{x},x]$, then $C_{\bar{x},x} \cap [\bar{x},y]=C_{\bar{x},y}$, and $p_{\bar{x},x}|_{[\bar{x},y]}=p_{\bar{x},y}$;
  \item For any $z \in [\bar{x},x]$ and $w \in p_{\bar{x},x}(z)$, we have $d(z,w) \leqslant l$;
  \item $\forall z\in [\bar{x},x]$, $\sharp \big(B(z,Ml) \cap C_{\bar{x},x}\big) \leqslant 3M$.
\end{itemize}

Suppose for any $\bar{x},x \in V$ with $\dim[\bar{x},x]\leqslant \rank-1$, we have defined $C_{\bar{x},x} \subseteq [\bar{x},x]$ and a map $p_{\bar{x},x}\colon [\bar{x},x] \rightarrow \mathcal{P}(C_{\bar{x},x})$ satisfying:
\begin{itemize}
  \item If $d(x,y)=1$ and $y \in [\bar{x},x]$, then $C_{\bar{x},x} \cap [\bar{x},y]=C_{\bar{x},y}$, and $p_{\bar{x},x}|_{[\bar{x},y]}=p_{\bar{x},y}$;
  \item For $z \in [\bar{x},x]$ and $w \in p_{\bar{x},x}(z)$, we have $d(z,w) \leqslant \frac{(\rank-1)\rank}{2}l$;
  \item $\forall z\in [\bar{x},x]$, $\sharp \big(B(z,Ml) \cap C_{\bar{x},x}\big) \leqslant (3M)^{\rank-1} (\rank-1)!$.
\end{itemize}
Now we focus on $\bar{x},x \in V$ with $\dim[\bar{x},x]=\rank$. For any $n \in \mathbb{N}$ with $nl \leqslant d_{nor}(\bar{x},x)$, by Proposition~\ref{dec final},
$$[\bar{x},x] \cap S_{nor}(\bar{x},nl)=\bigcup_{h\in H_{nl}^x}F_h^x=\bigcup_{h\in H_{nl}^x}[x_h,v_{nl}^x],$$
where $v_{nl}^x$ is the farthest point from $\bar{x}$ in $[\bar{x},x] \cap S_{nor}(\bar{x},nl)$, $H_{nl}^x$ is the set of hyperplanes crossing the $nl$-th cube of the normal cube path from $\bar{x}$ to $x$, and we also have $\dim [x_h,v_{nl}^x] < \dim[\bar{x},x]$. By induction, $C_{x_h,v_{nl}^x}$ and $p_{x_h,v_{nl}^x}$ have already been defined. Now we define
$$C_{\bar{x},x}^n=\bigcup _{h\in H_{nl}^x}C_{x_h,v_{nl}^x},$$
and
$$C_{\bar{x},x}=\bigcup_{n=0}^{\lfloor d_{nor}(\bar{x},x)/l \rfloor}C_{\bar{x},x}^n.$$
For any $z\in [\bar{x},x]$, let $\tilde{z}$ be the $nl$-th vertex on the normal cube path from $\bar{x}$ to $z$, where $n=\lfloor d_{nor}(\bar{x},z)/l \rfloor$, so $d_{nor}(\tilde{z},z)\leqslant l$, which implies $d(\tilde{z},z) \leqslant \rank l$, and
$$\tilde{z} \in [\bar{x},x] \cap S_{nor}(\bar{x},nl)=\bigcup_{h\in H_{nl}^x}[x_h,v_{nl}^x].$$
Now define
$$p_{\bar{x},x}(z)=\bigcup\big\{p_{x_h,v_{nl}^x}(\tilde{z}): h\in H_{nl}^x \mbox{~and~}\tilde{z}\in [x_h,v_{nl}^x]\big\},$$
and we need to verify the requirements hold for $C_{\bar{x},x}$ and $p_{\bar{x},x}$.

\textbf{\emph{First}}, suppose $d(x,y)=1$ and $y \in [\bar{x},x]$, and let $h'$ be the hyperplane separating $x$ from $y$. Given $n\in \mathbb{N}$ such that $[\bar{x},y] \cap S_{nor}(\bar{x},nl) \neq \emptyset$, by Proposition \ref{char for v}, $v_{nl}^x$ is the $nl$-th vertex on the normal cube path from $\bar{x}$ to $x$, and $v_{nl}^y$ is the $nl$-th vertex on the normal cube path from $\bar{x}$ to $y$. Due to the fellow-traveller property, Proposition~\ref{fellow-traveller}, $d(v_{nl}^x\, ,v_{nl}^y) \leqslant 1$. By Proposition \ref{def for v}, we have
$$v_{nl}^y \in S_{nor}(\bar{x},nl)\cap [\bar{x},y] \subseteq S_{nor}(\bar{x},nl)\cap [\bar{x},x] \subseteq [\bar{x},v_{nl}^x].$$
Recall that $H(z,w)$ denotes the set of all hyperplanes separating $z$ from $w$. Obviously,
$$H(\bar{x},x)=H(\bar{x},y) \cup \{h'\},$$
which implies $H_{nl}^y \subseteq H_{nl}^x \subseteq H_{nl}^y \cup \{h'\},$ by Lemma \ref{char nor dist} and Lemma \ref{char for H_n}.

If $h' \in H_{nl}^x$ then $F_{h'}^x \cap [\bar{x},y]=\emptyset$ by Proposition \ref{dec prel}. On the other hand, $\forall h\in H_{nl}^y$ by Lemma \ref{char for x_h}, $y_h$ is the unique point in $B_{nor}(\bar{x},nl)$ such that $h$ separates $\bar{x}$ from $y_h$, and for any hyperplane $k \pitchfork h$, $k$ does not separate $y$ from $\bar{x}$. This implies $y_h=x_h$ since $H_{nl}^y \subseteq H_{nl}^x$, so we can do induction for the new ``base" point $y_h=x_h$ and $v_{nl}^y, v_{nl}^x$, since $d(v_{nl}^x,v_{nl}^y) \leqslant 1$ and $v_{nl}^y \in [x_h,v_{nl}^x]$. This implies
$$C_{x_h,v_{nl}^x} \cap [x_h,v_{nl}^y]=C_{x_h,v_{nl}^y}.$$
Since $C_{x_h,v_{nl}^x} \subseteq [x_h,v_{nl}^x]$, we have
$$C_{x_h,v_{nl}^x} \cap [\bar{x},y]=C_{x_h,v_{nl}^x} \cap [x_h,v_{nl}^x] \cap [\bar{x},y]=C_{x_h,v_{nl}^x} \cap [x_h, \mu(x_h,v_{nl}^x,y)].$$
\textbf{Claim:} $\mu(x_h,v_{nl}^x,y)=v_{nl}^y$.
Indeed, if $v_{nl}^x=v_{nl}^y$, then it holds naturally; If $v_{nl}^x \neq v_{nl}^y$, then by Lemma \ref{consistency lemma}, $v_{nl}^x \notin [\bar{x},y]$. Since $d(v_{nl}^x,v_{nl}^y)=1$, so $v_{nl}^x \in [v_{nl}^y,y]$ or $v_{nl}^y \in [v_{nl}^x,y]$. While the former cannot hold since $[v_{nl}^y,y] \subseteq [\bar{x},y]$, so $v_{nl}^y \in [v_{nl}^x,y]$, which implies
$$v_{nl}^y \in [x_h,y] \cap [x_h,v_{nl}^x] \cap [v_{nl}^x,y],$$
i.e. $v_{nl}^y=\mu(x_h,v_{nl}^x,y)$.

By the claim,
$$C_{x_h,v_{nl}^x} \cap [\bar{x},y]=C_{x_h,v_{nl}^x} \cap [x_h, v_{nl}^y]=C_{x_h,v_{nl}^y}.$$
Now for the above $n$, we have
\begin{eqnarray*}
        C_{\bar{x},x}^n \cap [\bar{x},y] & = & \bigcup_{h\in H_{nl}^x}\big(C_{x_h,v_{nl}^x} \cap [\bar{x},y]\big)=\bigcup_{h\in H_{nl}^y}\big(C_{x_h,v_{nl}^x} \cap [\bar{x},y]\big) \\
                 & = & \bigcup_{h\in H_{nl}^y} C_{x_h,v_{nl}^y} =\bigcup_{h\in H_{nl}^y} C_{y_h,v_{nl}^y}= C_{\bar{x},y}^n.
\end{eqnarray*}
Since $C_{\bar{x},x}^n \subseteq [\bar{x},x] \cap S_{nor}(\bar{x},nl)$, we have
\begin{eqnarray*}
        C_{\bar{x},x} \cap [\bar{x},y] & = & \bigcup_{n=0}^{\lfloor d_{nor}(\bar{x},x)/l \rfloor}C_{\bar{x},x}^n \cap [\bar{x},y] = \bigcup_{n:[\bar{x},x] \cap S_{nor}(\bar{x},nl) \neq \emptyset} C_{\bar{x},x}^n \cap [\bar{x},y]\\
                 & = & \bigcup_{n:[\bar{x},y] \cap S_{nor}(\bar{x},nl) \neq \emptyset} C_{\bar{x},x}^n \cap [\bar{x},y] = \bigcup_{n:[\bar{x},y] \cap S_{nor}(\bar{x},nl) \neq \emptyset} C_{\bar{x},y}^n = C_{\bar{x},y}.
\end{eqnarray*}

$\forall z\in [\bar{x},y]$, one need to show that $p_{\bar{x},x}(z)=p_{\bar{x},y}(z)$. Let $\tilde{z}$ be the $nl$-th vertex on the normal cube path from $\bar{x}$ to $z$, where $n=\lfloor d_{nor}(\bar{x},z)/l \rfloor$. By the analysis above, we know
$$d(v_{nl}^x,v_{nl}^y) \leqslant 1, v_{nl}^y \in [v_{nl}^x,y], x_h=y_h, H_{nl}^y \subseteq H_{nl}^x \subseteq H_{nl}^y \cup \{h'\}.$$
For $h\in H_{nl}^x$ with $\tilde{z}\in [x_h,v_{nl}^x]$, then $h\in H_{nl}^y$, i.e. $h \neq h'$ since $\tilde{z}\in [\bar{x},y]$. Now for such $h$,
$$\tilde{z} \in [\bar{x},y] \cap [x_h,v_{nl}^x] = [x_h,v_{nl}^y]=[y_h,v_{nl}^y],$$
where the first equation comes from the claim above. Inductively, we know for such $h$,
$$p_{x_h,v_{nl}^x}(\tilde{z})=p_{y_h,v_{nl}^y}(\tilde{z}).$$
Now by definition,
\begin{eqnarray*}
        p_{\bar{x},x}(z) & = & \bigcup\big\{p_{x_h,v_{nl}^x}(\tilde{z}): h\in H_{nl}^x \mbox{~and~}\tilde{z}\in [x_h,v_{nl}^x]\big\}\\
                         & = & \bigcup\big\{p_{x_h,v_{nl}^x}(\tilde{z}): h\in H_{nl}^y \mbox{~and~}\tilde{z}\in [x_h,v_{nl}^x]\big\}\\
                         & = & \bigcup\big\{p_{y_h,v_{nl}^y}(\tilde{z}): h\in H_{nl}^y \mbox{~and~}\tilde{z}\in [y_h,v_{nl}^y]\big\}\\
                         & = & p_{\bar{x},y}(z).
\end{eqnarray*}

\textbf{\emph{Second}}, for any $z \in [\bar{x},x]$ and $w \in p_{\bar{x},x}(z)$, assume that $w\in p_{x_h,v_{nl}^x}(\tilde{z})$ for some $h\in H_{nl}^x$ and $\tilde{z}\in [x_h,v_{nl}^x]$ as in the definition. By induction, we know $d(\tilde{z},w) \leqslant \frac{(\rank-1)\rank}{2}l$ since $\dim[x_h,v_{nl}^x]\leqslant \rank-1$. So
$$d(z,w) \leqslant d(z,\tilde{z})+d(\tilde{z},w) \leqslant \rank l+\frac{(\rank-1)\rank}{2}l=\frac{\rank(\rank+1)}{2}l.$$

\textbf{\emph{Third}}, for any $z \in [\bar{x},x]$, consider $B(z,Ml) \cap C_{\bar{x},x}$. Suppose $n\in \mathbb{N}$ satisfying $B(z,Ml) \cap C_{\bar{x},x} \cap S_{nor}(\bar{x},nl) \neq \emptyset$, so $B(z,Ml) \cap \big( \bigcup_{h\in H_{nl}^x}[x_h,v_{nl}^x] \big) \neq \emptyset$, which means there exists some $h\in H_{nl}^x$ such that $B(z,Ml) \cap [x_h,v_{nl}^x] \neq \emptyset$. For such $n$ and $h$, let $z'=\mu(z,x_h,v_{nl}^x) \in [x_h,v_{nl}^x]$. Obviously,
$$B(z,Ml) \cap [x_h,v_{nl}^x] \subseteq B(z',Ml) \cap [x_h,v_{nl}^x].$$
By induction, we have
$$\sharp \big( B(z,Ml) \cap C_{x_h,v_{nl}^x} \big) \leqslant \sharp \big( B(z',Ml) \cap C_{x_h,v_{nl}^x} \big) \leqslant (3M)^{\rank-1} (\rank-1)!.$$
Now for the above $z$, there exist at most $3M$ values of $n$ such that $B(z,Ml) \cap [\bar{x},x] \cap S_{nor}(\bar{x},nl) \neq \emptyset$; and for such $n$, since $\sharp H_{nl}^x \leqslant \rank$, there exist at most $\rank$ hyperplanes $h$ such that $B(z,Ml) \cap [\bar{x},x] \cap [x_h,v_{nl}^x] \neq \emptyset$. So we have
\begin{eqnarray*}
        \sharp \big(B(z,Ml) \cap C_{\bar{x},x}\big) & \leqslant & \sum_{n:B(z,Ml) \cap [\bar{x},x] \cap S_{nor}(\bar{x},nl) \neq \emptyset} \sharp \big(B(z,Ml) \cap C_{\bar{x},x}^n \big)\\
                & \leqslant & \sum_{n: as~above} ~~\sum_{h\in H_{nl}^x} \sharp \big(B(z,Ml) \cap C_{x_h,v_{nl}^x} \big)\\
                & \leqslant & 3M \cdot \rank \cdot (3M)^{\rank-1} (\rank-1)! = (3M)^\rank \rank!.
\end{eqnarray*}

Now we take $K=\frac{(\rank-1)\rank}{2}$, and $N=(3M)^\rank \rank!=(3K+9\rank+9)^\rank \rank!$, then the lemma holds for these constants.
\end{proof}

\section{Coarse median spaces}
In this section, we discuss the coarse median case, and prove Theorem~\ref{thm b}. We fix a coarse median space $X$ with geodesic metric $\rho$ and coarse median $\mu$ with parameters $K,H$ and finite rank $\rank$. The definitions and notations are the same as in Section 2.4. According to Remark~\ref{median assup}, we also assume that the coarse median $\mu$ satisfies M1 and M2. We recall:

\begin{thm}[\cite{spakula2016coarse}]\label{sw thm}
Any geodesic uniformly locally finite coarse median space of finite rank and at most exponential growth has Property A.
\end{thm}

Our result, Theorem~\ref{thm b}, says that any coarse median space as above has subexponential asymptotic dimension growth. Thus, combining with Ozawa's result~\cite{ozawa2012metric},
our theorem yields a strengthening  of  Theorem~\ref{sw thm}.

To prove Theorem \ref{thm b}, we use several notations and lemmas from~\cite{spakula2016coarse}.
We use the notation $x\sim_s y$ for $\rho(x,y) \leqslant s$.
Given $r>0$ and $a,b\in X$, the coarse interval $[a,b]_r$ is defined to be:
$$[a,b]_r=\{z\in X: \mu(a,b,z)\sim_r z\}.$$
By a result of Bowditch~\cite{bowditch2014embedding}, there exists a constant $\lambda >0$ depending only on the parameter $K,H$, such that for all $x,y,z\in X$, $\mu(x,y,z)\in [x,y]_\lambda$.

Also recall that the median axiom M3 holds in the coarse median case up to a constant $\gamma \geqslant 0$ depending only on the parameters $K,H$: for all $x,y,z,u,v\in X$, we have
$$\mu(\mu(x,y,z),u,v)\sim_\gamma \mu(\mu(x,u,v), \mu(y,u,v), z).$$
Actually we can take $\gamma=3K(3K+2)H(5)+(3K+2)H(0)$.

Given $r,t,\kappa \geqslant 0$, denote
\begin{eqnarray*}
     &&L_1(r)=(K+1)r+K\lambda +\gamma +2H(0),\\
     &&L_2(r,\kappa)=(K+1)r+\kappa+H(0), \mbox{and}\\
     &&L_3(r,t)=3^\rank K^\rank rt+r.
\end{eqnarray*}

We need the following lemmas from \cite{spakula2016coarse}.

\begin{lem}[\cite{spakula2016coarse}]\label{lem1}
Let $X$ be a coarse median space, $r\geqslant 0$, and let $a,b\in X$, $x\in [a,b]_\lambda$. Then $[a,x]_r \subset [a,b]_{L_1(r)}$.
\end{lem}

\begin{lem}[\cite{spakula2016coarse}]\label{lem2}
Let $X$ be a geodesic coarse median space of rank at most $\rank$. For every $\kappa >0$ and $t>0$, there exists $r_t>0$, such that for all $r\geqslant r_t$, $a,b\in X$, there exists $h\in [a,b]_{L_1(r)}$, such that
\begin{itemize}
  \item $\rho(a,h) \leqslant L_3(r,t)$, and
  \item $B(a,rt) \cap [a,b]_\kappa \subset [a,h]_{L_2(r,\kappa)}$.
\end{itemize}
\end{lem}

\begin{lem}[\cite{spakula2016coarse}]\label{lem3}
Let $X$ be a coarse median space. Fix $\kappa >0$. There exist constants $\alpha,\beta \geqslant 0$ depending only on the parameters of the coarse median structure and $\kappa$, such that the following holds: let $a,b,h,m \in X$ and $r\geqslant 0$ satisfy $m\in [a,h]_{L_2(r,\kappa)}$, $h\in [a,b]_{L_1(r)}$. Then $p=\mu(m,b,h)$ satisfies $\rho(h,p) \leqslant \alpha r+ \beta$.
\end{lem}

\begin{proof}[Proof of Theorem \ref{thm b}]
The proof is based on the construction used in~\cite{spakula2016coarse} to prove property A, and for the readers' convenience, we give a sketch of their proof. In fact we will verify the stronger conditions on the $S$ sets required to apply Theorem~\ref{main prop}. Fix a base point $x_0 \in X$, and let $\alpha,\beta$ be the constants from Lemma \ref{lem3}. First apply Lemma \ref{lem2} for $\kappa = \lambda$ and all $t\in \mathbb{N}$ to obtain a sequence $r_t \in \mathbb{N}$, such that the conclusion of the lemma holds. Furthermore, we can choose the $r_t$ inductively to arrange the sequence $\mathbb{N} \ni t \mapsto l_t=\frac{tr_t-H(0)}{3K}$ is increasing.

Now fix $x\in X$, $t\in \mathbb{N}$, and $k\in \{1,2,\ldots, 3l_t\}$. For any $y\in B(x,k)$, Lemma \ref{lem2} applied for $a=y$, $b=x_0$ and $r=r_t$, produces a point $h_y\in [y,x_0]_{L_1(r_t)}$. We define
$$S(x,k,l_t)=\{h_y\in X: y\in B(x,k)\}.$$
We need to verify these sets satisfy Condition 2 in the statement of Proposition \ref{main prop}, i.e. we need to show there exists a subexponential function $f\colon \mathbb{R} \rightarrow \mathbb{R}$, satisfying:
\begin{enumerate}[i)]
  \item $\forall t \in \mathbb{N}$, $\exists S_t >0$, such that $S(x,k,l_t) \subseteq B(x,S_t)$ for all $k=1,\ldots, 3l_t$;
  \item $\forall t \in \mathbb{N}$, $\forall k,k'$ with $1 \leqslant k \leqslant k' \leqslant 3l_t$, $\forall x \in X$,
                we have $S(x,k,l_t) \subseteq S(x,k',l_t)$;
  \item $\forall x,y \in X$ with $\rho(x,y) \leqslant l_t$, we have:
    \begin{itemize}
      \item $S(x,k-\rho(x,y),l_t) \subset S(x,k,l_t) \cap S(y,k,l_t)$, for $k=\rho(x,y)+1,\ldots,3l_t$;
      \item $S(x,k+\rho(x,y),l_t) \supset S(x,k,l_t) \cap S(y,k,l_t)$, for $k=1,\ldots,3l_t-\rho(x,y)$;
    \end{itemize}
  \item $\forall t \in \mathbb{N}$, $\forall k=1,2,\ldots,3l_t$, $\forall x \in X$, we have $\sharp S(x,k,l_t) \leqslant f(l_t)$.
\end{enumerate}

By the construction, ii) and iii) hold naturally. For i), by Lemma \ref{lem2}, we know $S(x,k,l_t) \subseteq B(x,l_t+L_3(r_t,t))$. The only thing left is to find a subexponential function $f$ such that condition iv) holds. The following argument follows totally from the proof in~\cite{spakula2016coarse}, and we omit some calculation. The readers can turn to their original paper for more details.

Take $y\in B(x,k)$, with the notation as above. Denote $m_y=\mu(x,y,x_0)$. Then by Lemma \ref{lem2}, one can deduce that $m_y\in [y,h_y]_{L_2(r_t,\lambda)}$. Now since $h_y\in [y,x_0]_{L_1(r_t)}$, Lemma~\ref{lem3} implies the point $p_y=\mu(m_y,x_0,h_y)\in [m_y,x_0]_\lambda$ satisfies $\rho(h_y,p_y) \leqslant \alpha r_t + \beta$. As $m_y=\mu(x,y,x_0)\in[x,x_0]_\lambda$, Lemma \ref{lem1} now implies $p_y\in [x,x_0]_{L_1(\lambda)}$. Consequently, we have $\rho(x,p_y) \leqslant 3l_t+3^\rank K^\rank tr_t +r_t +\alpha r_t + \beta$, which depends linearly on $l_t$. Now by Proposition 9.8 in~\cite{bowditch2014embedding}, the number of possible points $p_y$ is bounded by $P(l_t)$ for some polynomial $P$ depending only on $H,K,\rank$ and uniform local finiteness of $X$. Since $X$ has at most exponential growth, it follows that $\sharp S(x,k,l_t)$ is at most $P(l_t)c'c^{r_t}$ for some constants $c,c'\geqslant 1$. Take $f(l_t)=P(l_t)c'c^{r_t}$ and recall that in the limit $r_t/l_t\rightarrow 0$.  We extend $f$ to a function on $\mathbb R^+$ by setting $f(r):=f(l_t)$ for $r\in (l_{t-1}, l_t]$. This completes the proof.

\end{proof}


\bibliographystyle{plain}
\bibliography{bibfileASDIM1}

\begin{thebibliography}{10}

\bibitem{behrstock2015asymptotic}
J.~Behrstock, M.F. Hagen, and A.~Sisto.
\newblock Asymptotic dimension and small-cancellation for hierarchically
  hyperbolic spaces and groups.
\newblock {\em {\rm arXiv:1512.06071}}, 2015.

\bibitem{behrstock2015hierarchically}
J.~Behrstock, M.F. Hagen, and A.~Sisto.
\newblock Hierarchically hyperbolic spaces ii: combination theorems and the
  distance formula.
\newblock {\em {\rm arXiv:1509.00632}}, 2015.

\bibitem{BD05}
G.~Bell and A.~N. Dranishnikov.
\newblock Asymptotic dimension in {B}\c{e}dlewo.
\newblock {\em Topology Proc}, 38:209--236, 2011.

\bibitem{bell2005growth}
G.~C. Bell.
\newblock Growth of the asymptotic dimension function for groups.
\newblock {\em {\rm arXiv:math/0504577}}, 2005.

\bibitem{bestvina2010asymptotic}
M.~Bestvina, K.~Bromberg, and K.~Fujiwara.
\newblock The asymptotic dimension of mapping class groups is finite.
\newblock {\em {\rm arXiv:1006.1939}}, 2010.

\bibitem{bowditch2013coarse}
B.~H. Bowditch.
\newblock Coarse median spaces and groups.
\newblock {\em Pacific Journal of Mathematics}, 261(1):53--93, 2013.

\bibitem{bowditch2014embedding}
B.~H. Bowditch.
\newblock Embedding median algebras in products of trees.
\newblock {\em Geometriae Dedicata}, 170(1):157--176, 2014.

\bibitem{BH99}
M.~R. Bridson and A.~H$\mathrm{\ddot{a}}$fliger.
\newblock {\em Metric spaces of non-positive curvature}, volume 319 of {\em
  Grundlehren der Mathematischen Wissenschaften [Fundamental Principles of
  Mathematical Sciences]}.
\newblock Springer-Verlag, Berlin, 1999.

\bibitem{brown2008c}
N.~P. Brown and N.~Ozawa.
\newblock {\em $C^*$-algebras and finite-dimensional approximations}, volume~88
  of {\em Graduate Studies in Mathematics}.
\newblock American Mathematical Society, Providence, RI, 2008.

\bibitem{campbell2005hilbert}
S.~Campbell and G.~A. Niblo.
\newblock Hilbert space compression and exactness of discrete groups.
\newblock {\em Journal of functional analysis}, 222(2):292--305, 2005.

\bibitem{chatterji2005wall}
I.~Chatterji and G.~A. Niblo.
\newblock From wall spaces to {CAT(0)} cube complexes.
\newblock {\em International Journal of Algebra and Computation},
  15(05n06):875--885, 2005.

\bibitem{chepoi2000graphs}
V.~Chepoi.
\newblock Graphs of some {CAT(0)} complexes.
\newblock {\em Advances in Applied Mathematics}, 24(2):125--179, 2000.

\bibitem{dranishnikov2006groups}
A.~N. Dranishnikov.
\newblock Groups with a polynomial dimension growth.
\newblock {\em Geometriae Dedicata}, 119(1):1--15, 2006.

\bibitem{gromov1987hyperbolic}
M.~Gromov.
\newblock Hyperbolic groups.
\newblock In {\em Essays in group theory}, pages 75--263. Springer, 1987.

\bibitem{gromov1992asymptotic}
M.~Gromov.
\newblock Asymptotic invariants of infinite groups.
\newblock {\em Geometric group theory, {V}ol.2, volume 182 of London Math. Soc.
  Lecture Note Ser.}, pages 1--295, 1993.

\bibitem{isbell1980median}
J.~R. Isbell.
\newblock Median algebra.
\newblock {\em Transactions of the American Mathematical Society},
  260(2):319--362, 1980.

\bibitem{kaimanovich2004boundary}
V.~A. Kaimanovich.
\newblock Boundary amenability of hyperbolic spaces.
\newblock In {\em Discrete geometric analysis}, volume 347 of {\em Contemp.
  Math.}, pages 83--111. Amer. Math. Soc., Providence, RI, 2004.

\bibitem{niblo1998geometry}
G.~A. Niblo and L.~D. Reeves.
\newblock The geometry of cube complexes and the complexity of their
  fundamental groups.
\newblock {\em Topology}, 37(3):621--633, 1998.

\bibitem{nica2004cubulating}
B.~Nica.
\newblock Cubulating spaces with walls.
\newblock {\em Algebr. Geom. Topol}, 4:297--309, 2004.

\bibitem{NY12}
P.~W. Nowak and G.~Yu.
\newblock {\em Large scale geometry}.
\newblock 2012.

\bibitem{oppenheim2014depth}
I.~Oppenheim.
\newblock An intermediate quasi-isometric invariant between subexponential
  asymptotic dimension growth and {Y}u's property {A}.
\newblock {\em Internat. J. Algebra Comput.}, 24(6):909--922, 2014.

\bibitem{ozawa2012metric}
N.~Ozawa.
\newblock Metric spaces with subexponential asymptotic dimension growth.
\newblock {\em International Journal of Algebra and Computation},
  22(02):1250011, 2012.

\bibitem{reeves1995biautomatic}
L.~D. Reeves.
\newblock {\em Biautomatic structures and combinatorics for cube complexes}.
\newblock PhD thesis, University of Melbourne, 1995.

\bibitem{roe2005hyperbolic}
J.~Roe.
\newblock Hyperbolic groups have finite asymptotic dimension.
\newblock {\em Proceedings of the American Mathematical Society},
  133(9):2489--2490, 2005.

\bibitem{roller1998poc}
M.~Roller.
\newblock Poc sets, median algebras and group actions. an extended study of
  {D}unwoody's construction and {S}ageev's theorem.
\newblock {\em Southampton Preprint Archive}, 1998.

\bibitem{sageev1995ends}
M.~Sageev.
\newblock Ends of group pairs and non-positively curved cube complexes.
\newblock {\em Proceedings of the London Mathematical Society}, 3(3):585--617,
  1995.

\bibitem{smith2007asymptotic}
J.~Smith.
\newblock The asymptotic dimension of the first {G}rigorchuk group is infinity.
\newblock {\em Revista Matem{\'a}tica Complutense}, 20(1):119--121, 2007.

\bibitem{spakula2016coarse}
J.~\u{S}pakula and N.~J. Wright.
\newblock Coarse medians and property {A}.
\newblock {\em {\rm arXiv:1602.06084}}, 2016.

\bibitem{wright2012finite}
N.~J. Wright.
\newblock Finite asymptotic dimension for {CAT(0)} cube complexes.
\newblock {\em Geometry \& Topology}, 16(1):527--554, 2012.

\bibitem{Yu98}
G.~Yu.
\newblock The {N}ovikov conjecture for groups with finite asymptotic dimension.
\newblock {\em The Annals of Mathematics}, 147(2):325--355, 1998.

\bibitem{yu2000coarse}
G.~Yu.
\newblock The coarse {B}aum-{C}onnes conjecture for spaces which admit a
  uniform embedding into {H}ilbert space.
\newblock {\em Inventiones Mathematicae}, 139(1):201--240, 2000.

\bibitem{zeidler2013coarse}
R.~Zeidler.
\newblock {Coarse median structures on groups}.
\newblock Master's thesis, University of Vienna, Vienna, Austria, 2013.

\end{thebibliography}

\end{document}